\documentclass{scrartcl}

\usepackage{amsmath, amssymb, amsthm, marginnote, ltxcmds, adjustbox, stmaryrd, graphicx, bbold, extarrows}
\usepackage{esvect}

\usepackage{mathtools, stackengine, changepage, ragged2e}

\usepackage{todonotes}
\usepackage{float}
\usepackage{multirow}

\definecolor{cite}{HTML}{11871E}
\definecolor{url}{HTML}{698996}
\definecolor{link}{HTML}{912F1B}

\usepackage[pdfencoding=unicode, colorlinks=true, linkcolor=link, citecolor=cite, urlcolor=url, linktocpage]{hyperref}

\usepackage{subfig}

\usepackage[backend=biber, style=alphabetic, maxnames=10, maxalphanames=3, minalphanames=3]{biblatex}
\usepackage{tikz}
\usetikzlibrary{cd}
\usetikzlibrary{arrows, arrows.meta, positioning, calc}
\tikzcdset{arrow style=tikz, diagrams={>={Straight Barb[scale=0.8]}}}

\tikzstyle{arrow} = [-{Straight Barb[scale=0.8]}, line width=0.2mm]
\tikzset{
math to/.tip={Glyph[glyph math command=rightarrow]},
loop/.tip={Glyph[glyph math command=looparrowleft, swap]},
}

\usepackage{enumitem}
\newenvironment{myenum}[1]{%
\begin{enumerate}[label=#1,topsep=3pt,itemsep=2pt,partopsep=1pt,parsep=1pt]%
}%
{\end{enumerate}%
}

\usepackage[
	letterpaper,
	twoside=false,
	textheight=22cm,
	textwidth=14.4cm,
	marginparsep=0.75cm,
	marginparwidth=2.5cm,
	heightrounded,
	centering
]{geometry}

\usepackage{lineno}

\usepackage{contour}
\usepackage{ulem}

\contourlength{0.5pt}

\newcommand{\myuline}[1]{%
  \uline{\phantom{#1}}%
  \llap{\contour{white}{#1}}%
}

\makeatletter
\newcommand*{\saved@myuline}{}
\let\saved@myuline\myuline

\newcommand*{\mathuline}{%
  \mathpalette{\math@myuline\saved@myuline}%
}
\newcommand*{\math@myuline}[3]{%
  \mbox{#1{$#2#3\m@th$}}%
}

\renewcommand*{\myuline}{%
  \relax  
  \ifmmode
    \expandafter\mathuline
  \else
    \expandafter\saved@myuline
  \fi
}
\makeatother

\usepackage{libertine}
\usepackage{stmaryrd}

\usepackage[capitalise]{cleveref}

\Crefname{prop}{Proposition}{Propositions}
\Crefname{lem}{Lemma}{Lemmas}
\Crefname{cor}{Corollary}{Corollaries}
\Crefname{thm}{Theorem}{Theorems}
\Crefname{alphThm}{Theorem}{Theorems}

\Crefname{defn}{Definition}{Definitions}
\Crefname{notation}{Notation}{Notations}
\Crefname{cons}{Construction}{Constructions}
\Crefname{rmk}{Remark}{Remarks}
\Crefname{obs}{Observation}{Observations}
\Crefname{trick}{Trick}{Tricks}
\Crefname{warning}{Warning}{Warnings}
\Crefname{conj}{Conjecture}{Conjectures}
\Crefname{assump}{Assumption}{Assumptions}
\Crefname{recollect}{Recollection}{Recollections}
\Crefname{terminology}{Terminology}{Terminologies}
\Crefname{conditionsec}{Condition}{Conditions}
\Crefname{fact}{Fact}{Facts}

\Crefname{question}{Question}{Questions}
\Crefname{example}{Example}{Examples}

\Crefname{figure}{Figure}{Figures}

\crefformat{equation}{(#2#1#3)}
\crefformat{section}{\S#2#1#3}
\crefmultiformat{section}{\S\S#2#1#3}{and~#2#1#3}{, #2#1#3}{, and~#2#1#3}

\newtheorem{thm}[subsubsection]{Theorem}
\newtheorem{prop}[subsubsection]{Proposition}
\newtheorem{lem}[subsubsection]{Lemma}
\newtheorem{cor}[subsubsection]{Corollary}

\newtheorem{alphThm}{Theorem}

\newcommand{\neutralize}[1]{\expandafter\let\csname c@#1\endcsname\count@}
\makeatother

\newtheorem*{thm*}{Theorem}
\newtheorem*{prop*}{Proposition}
\newtheorem*{lem*}{Lemma}
\newtheorem*{cor*}{Corollary}
  
\newtheorem{alphConj}{Conjecture}



\newtheorem{alphCor}{Corrollary}



\newtheorem{alphProp}{Proposition}



\newtheorem{thmsec}{Theorem}[section]

\newtheorem{corsec}[thmsec]{Corollary}

\theoremstyle{definition}
\newtheorem*{defn*}{Definition}
\newtheorem{defn}[subsubsection]{Definition}
\newtheorem{cons}[subsubsection]{Construction}
\newtheorem{nota}[subsubsection]{Notation}

\newtheorem{recollect}[subsubsection]{Recollections}

\newtheorem{example}[subsubsection]{Example}

\newtheorem{rmk}[subsubsection]{Remark}
\newtheorem{obs}[subsubsection]{Observation}
\newtheorem{fact}[subsubsection]{Fact}

\theoremstyle{definition}

\newtheorem{questionsec}[thmsec]{Question}
\newtheorem{conditionsec}[thmsec]{Condition}

\newcommand{\terminalTCat}{\underline{\ast}}

\newcommand{\Pic}{\mathcal{P}\mathrm{ic}}
\newcommand{\susps}{\Sigma^{\infty}}

\newcommand{\spc}{\mathcal{S}}

\newcommand{\ind}{\mathrm{Ind}}
\DeclareMathOperator{\coind}{Coind}

\newcommand{\canonical}{\mathrm{can}}

\newcommand{\cat}{\mathrm{Cat}}

\newcommand{\op}{^{\mathrm{op}}}

\DeclareMathOperator{\mapsp}{map}
\newcommand{\sphere}{\mathbb{S}}

\DeclareMathOperator{\res}{Res}
\DeclareMathOperator{\mackey}{Mack}
\DeclareMathOperator{\presheaf}{Psh}
\newcommand{\orbit}{\mathcal{O}}
\newcommand{\singorbit}{\orbit^{\mathrm{sing}}}

\newcommand{\spectra}{\mathrm{Sp}}
\newcommand{\finite}{\mathrm{Fin}}
\newcommand{\inflated}{\mathrm{infl}}

\newcommand{\id}{\mathrm{id}}
\DeclareMathOperator{\map}{Map}
\DeclareMathOperator{\fib}{fib}
\DeclareMathOperator{\cofib}{cofib}
\DeclareMathOperator{\func}{Fun}

\newcommand{\module}{\mathrm{Mod}}

\newcommand{\stmodSmall}{\mathrm{stmod}}
\newcommand{\proper}{\mathcal{P}}

\DeclareMathOperator{\out}{Out}

\newcommand{\ambi}[3]{{{#1} \cap_{#2} {#3}}}

\newcommand{\udlcatC}{\udl{\category{C}}}
\newcommand{\udlcatD}{\udl{\category{D}}}

\newcommand{\udl}[1]{\underline{{#1}}}

\def\colim{\qopname\relax m{colim}}

\newcommand{\category}[1]{\mathcal{#1}}
\DeclareMathOperator{\spancategory}{Span}

\newcommand{\bbZ}{\mathbb{Z}}

\DeclareMathOperator{\infl}{infl}

\newcommand{\beckChevalley}{\mathrm{BC}}

\newcommand{\induced}{\mathrm{ind}}

\newcommand{\arrdisp}{0.33ex}
\newcommand{\arrdisplacementsp}{0.72ex}

\newcommand{\ardis}{\ar@<\arrdisp>}
\newcommand{\ardissp}{\ar@<\arrdisplacementsp>}

\newcommand{\classifyingspace}[2]{\myuline{E_{#2} #1}}
\newcommand{\cptfamily}{\mathrm{Cpt}}
\newcommand{\quotientspace}[2]{\myuline{B_{#1} #2}}

\newcommand{\uniquemap}[1]{#1}
\newcommand{\singpart}[1]{{#1}^{>1}}

\newcommand*\cocolon{%
        \nobreak
        \mskip6mu plus1mu
        \mathpunct{}%
        \nonscript
        \mkern-\thinmuskip
        {:}%
        \mskip2mu
        \relax
}



\title{Equivariant Poincar\'{e} duality for cyclic groups of prime order and the Nielsen realisation problem}
\author{\textsc{Kaif Hilman}\thanks{kaif@mpim-bonn.mpg.de} \and \textsc{Dominik Kirstein}\thanks{kirstein@mpim-bonn.mpg.de} \and \textsc{Christian Kremer}\thanks{kremer@mpim-bonn.mpg.de}}
\date{\today}

\begin{document}
\maketitle
\begin{abstract}
    In this companion article to \cite{HKK_PD}, we apply the theory of equivariant Poincar\'e duality  developed there in the special case of cyclic groups $C_p$ of prime order to remove, in a special case, a technical condition given by Davis--L\"uck \cite{davis2024nielsen} in their work on the Nielsen realisation problem for aspherical manifolds. Along the way, we will also give a complete characterisation of $C_p$--Poincar\'e spaces as well as introduce a genuine equivariant refinement of the classical notion of virtual Poincar\'e duality groups which might be of independent interest.
\end{abstract}

\begin{figure}[h]
\centering
\includegraphics[width=8cm]{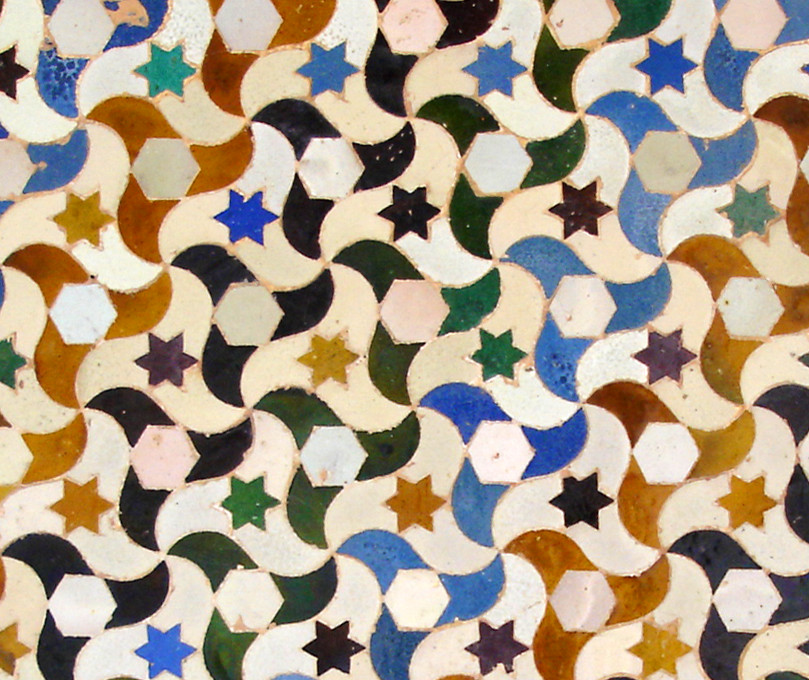}
\caption{A wall tiling at the Alhambra, in Granada, Spain, with symmetry group ``$p3$"\protect\footnotemark}
\label{fig:p3}
\end{figure}
\footnotetext{By Dmharvey, see https://commons.wikimedia.org/wiki/File:Alhambra-p3-closeup.jpg, license CC BY-SA 3.0}

\newpage

\section{Introduction}

A famous question due to Jakob Nielsen \cite{Nielsen_Original} in  geometric topology is the following: can any finite subgroup $G \subset \pi_0 \mathrm{hAut}(\Sigma_g)$  be lifted to an actual continuous group action on $\Sigma_g$, for $\Sigma_g$ a closed oriented surface of genus $g \geq 0$? This turns out to be possible, with Nielsen settling the case of $G$ finite cyclic, and Kerckhoff \cite{Kerckhoff_Nielsen} the general case.

In high-dimensions, asking for  too direct a generalisation of Nielsen's question inevitably results in wrong statements as there is simply no reason why a general homotopy equivalence $h \colon M \rightarrow M$ should be homotopic to any homeomorphism. However, rigidity phenomena in the theory of closed aspherical manifolds - closed connected manifolds with contractible universal covers - give some hope in generalising Nielsen's and Kerckhoff's results in this direction. This question may thus fairly be called the ``generalised Nielsen realisation problem for aspherical manifolds'' and has been investigated quite intensively, see for example \cite{Raymond-Scott_Failure_of_Nielsen,davisLeary,Block-Weinberger,lueck2022brown,davis2024nielsen}.\footnote{For completeness, we  mention here that there is also a large body of work on the Nielsen realisation problem for not necessarily aspherical 4--manifolds, c.f. for instance \cite{FarbLooijenga,BaragliaKonno,LeeDelPezzoAdvances,KonnoSpinAGT}, which has a much more geometric flavour.}

Unfortunately, even the hypothesis of closed aspherical manifolds is not  quite adequate, and we refer to \cite{Weinberger_Variations_on_a_theme_of_Borel} for a delightful survey of counterexamples.
Nevertheless, it turns out to be quite easy to dodge all potential reasons for counterexamples (for example, the failure for the existence of necessary group extensions due to Raymond--Scott \cite{Raymond-Scott_Failure_of_Nielsen}) by asking a slight variation of the generalised Nielsen problem:

\begin{questionsec}
    \label{quest:nielsen_via_extensions}
    Let $M$ be an aspherical manifold with fundamental group $\pi$ and consider an extension of groups
    $1 \rightarrow \pi \rightarrow \Gamma \rightarrow G \rightarrow 1$
    where $G$ is finite of odd order\footnote{Taking $G$ to be of odd order implies that certain $\mathrm{UNil}$-valued obstructions vanish, see \cite[Thm. 1.16.]{davis2024nielsen} or \cite[Sec. 6.4.]{Weinberger_Variations_on_a_theme_of_Borel}.}.
    Does the $\pi$-action on the universal cover $\widetilde{M}$ of $M$ extend to a $\Gamma$-action such that
    \[ \widetilde{M}^H \simeq \begin{cases}
        * \hspace{2mm }\text{if $H\leq \Gamma$, and $H$ is finite};\\
        \emptyset \hspace{2mm }\text{if $H\leq \Gamma$, and $H$ is infinite}?
    \end{cases} \]
    Equivalently, does the $\pi$-action on the universal cover of $M$ extend to a $\Gamma$-action in a way such that the resulting $\Gamma$-space models $\classifyingspace{\finite}{\Gamma}$, the universal space for proper $\Gamma$-actions?
\end{questionsec}

Provided the answer to \cref{quest:nielsen_via_extensions} is yes, one may construct a $G$-action on $M$ by using the residual action on $\pi \backslash \widetilde{M}$. For an account of the relation of \cref{quest:nielsen_via_extensions} to the generalized Nielsen realization problem in terms of homomorphisms $G \rightarrow \pi_0(\mathrm{hAut}(M)) \cong \mathrm{Out}(\pi_1(M))$, we refer the reader to the introduction of \cite{davis2024nielsen}.

\vspace{1mm}

In this article, we give a positive answer to \cref{quest:nielsen_via_extensions} in the very special situation when $M$ is high-dimensional, $\pi$ is hyperbolic, $G=C_p$ for $p$ odd, and if the extension is what we call \textit{pseudofree}, i.e. if each nontrivial finite subgroup $F \subset \Gamma$ satisfies $N_\Gamma F = F$.
Geometrically, this predicts that any $\Gamma$--manifold model $\widetilde{M}$ must have discrete fixed points (see \cref{rmk:Lueck_Weiermann_example}), whence the name. One of our main results is the following:

\begin{alphThm}\label{thm:existence_nielsen_realisation}
    Consider a group extension
    \begin{equation}
        \label{eq:extension_dl}
        1 \rightarrow \pi \rightarrow \Gamma \rightarrow C_p \rightarrow 1
    \end{equation}
    for an odd prime p. Suppose that
    \begin{myenum}{(\arabic*)}
        \item $\pi = \pi_1(M)$ for a closed orientable\footnote{Orientability is assumed only to simplify the exposition, and can be removed with some care.} aspherical manifold $M$ of dimension at least 5,
        \item $\pi$ is hyperbolic,
        \item $\Gamma$ is pseudofree.
    \end{myenum}
    Then there exists a cocompact
$\Gamma$-manifold model for $\classifyingspace{\finite}{\Gamma}$.
\end{alphThm}

To the best of our knowledge, the most general existence result for manifold models for $\classifyingspace{\finite}{\Gamma}$ that does not refer to specific differential geometric constructions  is due to Davis-L\"uck \cite[Thm. 1.16]{davis2024nielsen}, whose methods are mainly surgery-- and K--theoretic.
They prove \cref{thm:existence_nielsen_realisation} under an additional necessary group homological ``Condition (H)'' (c.f. \cref{nota:condition_H}) on $\Gamma$ which is previously considered mysterious.
This Condition (H) was discovered by L\"uck in \cite{lueck2022brown} as necessary for the existence of manifold models for $\classifyingspace{\finite}{\Gamma}$, but was also used to construct certain models for $\classifyingspace{\finite}{\Gamma}$ which satisfy  some kind of equivariant Poincar\'e duality.
Davis-L\"uck then show in which situations these can actually be turned into equivariant manifolds.
Condition (H), however, seems complicated and hard to verify.
Our main contribution to this problem is to show that, in the situation of \cref{thm:existence_nielsen_realisation}, Condition (H) is actually automatic, and we achieve this by locating it in the more conceptual context of \textit{equivariant Poincar\'e duality} as developed in \cite{HKK_PD}. We hope that these techniques will allow us to go beyond the pseudofree situation where Davis-L\"uck applies, so that in the future we might be able to construct group actions on aspherical manifolds with nondiscrete fixed point sets. As will be clear later, our main input to remove Davis-L\"uck's Condition (H), \cref{thm:genuine_virtual_duality_recognition},   does not refer to discrete fixed points at all.  

To round off our commentary on \cref{thm:existence_nielsen_realisation}, it should also be noted that, in principle, the theorem reduces the geometric problem to a purely algebraic one of producing the appropriate group extension under the given hypotheses on $p$ and $\pi$.
As explained e.g. in \cite[p.1]{davis2024nielsen}, there is an obstruction measuring when a homomorphism $C_p \to \out(\pi)$ is induced by an extension \cref{eq:extension_dl}.
It vanishes for hyperbolic groups as they have trivial center.
For the Nielsen realisation problem, \cref{thm:existence_nielsen_realisation} thus has the following implication.

\begin{corsec}
    Let $M$ be a closed orientable aspherical manifold with hyperbolic fundamental group of dimension at least $5$, $p$ an odd prime, and $\alpha \colon C_p \rightarrow \out(\pi_1M)$ a homomorphism. Then the Nielsen realisation problem for $\alpha$ admits a solution, provided the associated extension $\Gamma$ is pseudofree.
\end{corsec}

Before moving on to elaborate on equivariant Poincar\'e duality as used in this work, we first state the aforementioned  Condition (H) and recall the argument of \cite[Lemma 1.9]{lueck2022brown} as to why it is  necessary for the conclusion of \cref{thm:existence_nielsen_realisation} to hold.  This shows that Condition (H) is not merely an artefact of the proof strategy of \cite{davis2024nielsen} but is rather a point that must be dealt with in one way or another.

For a pseudofree extension \cref{eq:std_extension}, a result of L\"uck--Weiermann \cite{LWClassifying} (c.f. \cref{thm:lueck_weiermann}) shows that the subspace $\classifyingspace{\finite}{\Gamma}^{>1}$ of points in $\classifyingspace{\finite}{\Gamma}$ with nontrivial isotropy is discrete, more precisely, $\classifyingspace{\finite}{\Gamma}^{>1} \simeq \coprod_{F \in \mathcal{M}} \Gamma/F$ where $F$ runs through a set of representatives of conjugacy classes of nontrivial finite subgroups. Writing $H^\Gamma_*(X) \coloneqq H_*(X_{h\Gamma};\mathbb{Z})$ for the integral Borel homology of a space $X$ with $\Gamma$--action, the condition may be stated as:

\begin{conditionsec}[Condition (H)]\label{nota:condition_H}
    For each finite subgroup $F\neq 1$ of $\Gamma$, the composite
    \begin{equation*}
        H_d^\Gamma(\classifyingspace{\finite}{\Gamma}, \classifyingspace{\finite}{\Gamma}^{>1}) 
        \xrightarrow{\partial} H^\Gamma_{d-1}(\classifyingspace{\finite}{\Gamma}^{>1}) 
        \simeq \bigoplus_{F' \in \mathcal{M}} H_{d-1}(BF')
        \xrightarrow{\mathrm{proj}_F} H_{d-1}(BF)
    \end{equation*}
    is surjective.
\end{conditionsec}

To see why condition (H) is necessary for the existence problem, suppose that there exists a $d$-dimensional cocompact manifold model $N$ for $\classifyingspace{\finite}{\Gamma}$.
Let us assume for simplicity that $N$ is smooth and that $\Gamma$ acts smoothly preserving the orientation.
As mentioned before, the singular part $N^{>1}$ of $N$ of points with nontrivial isotropy is discrete if the extension is pseudofree.
Denote by $Q$ the complement of an equivariant tubular neighbourhood of $N^{>1}$ in $N$ with boundary $\partial Q$.
Then the $\Gamma$-action on $Q$ is free and the quotient pair $(\Gamma \backslash Q, \Gamma \backslash \partial Q)$ is a compact $d$-manifold with boundary. See \cref{fig:discs} for an illustration.
Thus, for every path component $L$ of $\Gamma\backslash\partial Q$, we obtain the commutative diagram
\begin{equation*}
\begin{tikzcd}
    H_d(\Gamma \backslash Q,\Gamma \backslash \partial Q) \ar[r] \ar[d,"\simeq"] 
    & H_{d-1}(\Gamma \backslash \partial Q) \ar[d]  \ar[r, "\mathrm{proj}",two heads]
    & H_{d-1}(L) \ar[d, two heads] 
    \\
    H_d^\Gamma(\classifyingspace{\finite}{\Gamma}, \classifyingspace{\finite}{\Gamma}^{>1}) \ar[r] 
    & H^\Gamma_{d-1}(\classifyingspace{\finite}{\Gamma}^{>1})  \ar[r, "\mathrm{proj}"]
    & H_{d-1}(BF).
\end{tikzcd}
\end{equation*}
The left vertical arrow is an equivalence using excision and that homology of $(\Gamma \backslash Q, \Gamma \backslash \partial Q)$ agrees with Borel homology of $(Q, \partial Q)$ as the $\Gamma$-action is free.
The fundamental class of $(\Gamma \backslash Q, \Gamma \backslash \partial Q)$ gets sent to a fundamental class of each boundary component along the upper composite, and so the top composite is surjective. Moreover, note that each component $L$ of $\Gamma \backslash \partial Q$ is obtained as the quotient of a sphere by a free action of the isotropy group $F$ of the corresponding fixed point in $N^{>1}$. 
Recall that for any free $F$-action on a $(d-1)$--sphere $S$ for a finite group $F$, the map $F \backslash S \rightarrow BF$ is is $(d-1)$--connected and induces a surjection on homology up to degree $d-1$ so the right vertical map is surjective.
Together this shows that the bottom composite is surjective in each component.

\begin{figure}[h]
\label{fig:discs}
\centering
\includegraphics[width=4cm]{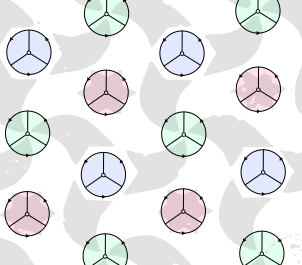}
\caption{Discs around the singular part for the symmetry group p3}
\end{figure}

\subsection{Equivariant Poincar\'e duality}
Equivariant Poincar\'e duality is fundamentally about understanding group actions on manifolds. The notion of a $G$-equivariant Poincar\'e complex is designed to satisfy more or less all the homological or cohomological constraints that a smooth $G$-manifold satisfies. In particular, satisfying equivariant Poincar\'e duality can \textit{obstruct} the existence of certain group actions on manifolds. This philosophy is old and has been quite successful, and we exploited it in \cite{HKK_PD} to generalise   some classical nonexistence results\footnote{See also the references therein for more results of this type.} using new methods. 

In this article, we want to carry across a different point: equivariant Poincar\'e duality does not only obstruct, but is also quite useful to \textit{construct} group actions on manifolds. The testing ground we chose to demonstrate our claim is \cref{quest:nielsen_via_extensions}. Here, we will only employ the theory of equivariant Poincar\'e duality for the group $G=C_p$, and since this case is much simpler than for general compact Lie groups, we hope that it will demystify the  more abstract discussions in \cite{HKK_PD}. In particular, we can keep the level of  equivariant stable homotopy theory used throughout at a minimum, while still showing  some standard manipulations. We hope that readers with an interest in geometric topology and homotopy theory might find this to be a useful first exposition to categories of genuine $G$-spectra and their uses.

\subsubsection*{Equivariant Poincar\'e duality for the group $C_p$}

Our first theoretical goal is to give a characterisation of $C_p$-Poincar\'e duality that is adapted to our application on the Nielsen realisation problem.
In \cite{HKK_PD} we showed that if $\udl{X}$ is a $C_p$-space, then the following hold:
\begin{myenum}{(\arabic*)}
    \item If $\udl{X}$ is $C_p$-Poincar\'e, then the underlying space $X^{e}$ and the fixed points $X^{C_p}$ are nonequivariantly Poincar\'e (c.f. \cite[Thm. C]{HKK_PD}).
    \item Even if $\udl{X}$ is assumed to be a compact (i.e. equivariantly finitely dominated) $C_p$-space, requiring $X^{e}$ and $X^{C_p}$ to be nonequivariantly Poincar\'e is not sufficient to guarantee that $\udl{X}$ is $C_p$-Poincar\'e (c.f. \cite[Cor. 5.1.16]{HKK_PD}).
\end{myenum}
In this article, we identify the precise additional condition needed to ensure that $\udl{X}$ is $C_p$--Poincar\'e in the situation of (2).

\begin{alphThm}[c.f. \cref{thm:characterization_of_Cp_PD}]
    \label{mainthm:recognition}
    Let $\udl{X}$ be a compact $C_p$-space. 
    Denote by $\varepsilon \colon X^{C_p} \rightarrow X^{e}$ the inclusion of the fixed points, and assume that both $X^{C_p}$ and $X^{e}$ are nonequivariantly Poincar\'e. Let $D_{X^{C_p}}\in \func(X^{C_p},\spectra^{BC_p})$ be the dualising sheaf of the fixed point space. Then $\udl{X}$ is $C_p$-Poincar\'e if and only if the cofibre of the adjunction unit morphism in $\func(X^{C_p},\spectra^{BC_p})$
    \[ D_{X^{C_p}} \rightarrow \varepsilon^* \varepsilon_! D_{X^{C_p}} \]
    pointwise lies in $\spectra^{BC_p}_\ind\subseteq \spectra^{BC_p}$, the stable full subcategory generated by the image of $\ind^{C_p}_e\colon \spectra \rightarrow \spectra^{BC_p}$. 
\end{alphThm}

More details on the terms appearing in the theorem may be found in the material leading up to \cref{thm:characterization_of_Cp_PD}. In effect, this result gives an interpretation of the genuine equivariant notion of Poincar\'e duality purely in terms of nonequivariant and Borel equivariant properties. The method of proof is based on various cellular manoeuvres in equivariant stable homotopy theory developed in \cref{section:PD_for_C_p}, which might be of independent interest. Armed with this characterisation, we may now return to the modified Nielsen \cref{quest:nielsen_via_extensions} as we explain next.

\subsubsection*{Genunine virtual Poincar\'e duality groups and the proof of \cref{thm:existence_nielsen_realisation}}

Suppose we are given an extension of groups
\begin{equation}
    \label{eq:std_extension}
    1 \rightarrow \pi \rightarrow \Gamma \rightarrow C_p \rightarrow 1
\end{equation}
so that there exists a cocompact $\Gamma$-manifold $N$ modeling the $\Gamma$-homotopy type $\classifyingspace{\finite}{\Gamma}$. Under reasonable assumptions, we might expect that the $C_p$-space $\pi \backslash N$ is $C_p$-Poincar\'e. Motivated by this expectation, we will call $\Gamma$ a \textit{genuine virtual Poincar\'e duality group} if $\pi \backslash \classifyingspace{\finite}{\Gamma}$ is a $C_p$-Poincar\'e space. In fact, we define the notion of a genuine virtual Poincar\'e duality group in a broader context in \cref{section:gvpd_groups} and we hope that it can be a useful supplement to the classical theory of virtual Poincar\'e duality groups that appear for example in \cite{Brown_Cohomology_of_Groups}. In any case, using \cref{mainthm:recognition}, we will show our main result:

\begin{alphThm}[c.f. \cref{thm:extension_by_Cp_are_gvd}]\label{thm:genuine_virtual_duality_recognition}
    For the extension \cref{eq:std_extension}, assume that $\classifyingspace{\finite}{\Gamma}$ is compact. If
    \begin{myenum}{(\arabic*)}
        \item  for each nontrivial finite subgroup $F \subset \Gamma$, the group $W_\Gamma F$ is a Poincar\'e duality group;
        \item the group $\pi$ is a Poincar\'e duality group,
    \end{myenum}
    then $\Gamma$ is a genuine virtual Poincar\'e duality group.
\end{alphThm}

The significance of this result is that we relate the new notion of genuine virtual Poincar\'e duality groups, which enjoys good conceptual  properties, with the classical notion of Poincar\'e duality groups, which is easier to check. It will turn out that, in the situation of \cref{thm:existence_nielsen_realisation}, the group $\Gamma$ satisfies the conditions of \cref{thm:genuine_virtual_duality_recognition}, and so we see that $\pi \backslash \classifyingspace{\finite}{\Gamma}$ is $C_p$-Poincar\'e. In particular, this opens up the way for an equivariant fundamental class analysis (c.f. \cite[\textsection 4.5]{HKK_PD}) on the problem at hand, yielding the following:

\begin{alphThm}
\label{thm:condition_H_automatic}
    Let $\Gamma$ be as in \cref{thm:existence_nielsen_realisation}.
    Then $\Gamma$ satisfies Condition (H).
\end{alphThm}

A proof of this, and the more general \cref{prop:property_H}, will be given in \cref{sec:property_H}. Taking this for granted for the moment, we may now provide the proof of \cref{thm:existence_nielsen_realisation}.

\begin{proof}[Proof of \cref{thm:existence_nielsen_realisation}]
    This is now a direct consequence of \cite[Theorem 1.4]{davis2024nielsen}.
    The group theoretic assumptions therein are satisfied by \cref{lem:decompositon_of_fixed_points}, whereas Condition (H) is shown to hold in \cref{thm:condition_H_automatic}.
\end{proof}

\subsection*{Structural overview}

We recall in \cref{section:recollections}  some notions and constructions from the theory of equivariant Poincar\'e duality  \cite{HKK_PD} that will be pertinent to our purposes. Next, in \cref{section:PD_for_C_p}, we work towards proving \cref{mainthm:recognition}, and for this, it will be necessary to develop some theory on compact objects in $C_p$--genuine spectra. This we do in \cref{sec:cellular_manouvres,sec:compact_cp_spectra_characterisation}, which might be of independent interest. Having set up the requisite basic theory, we return to the problem at hand and define the notion of genuine virtual Poincar\'e duality groups in \cref{section:gvpd_groups}, refining the classical notion of virtual Poincar\'e duality groups. Therein, we will also prove a characterisation result tailored to our needs. Finally, we put together all the elements and prove \cref{thm:genuine_virtual_duality_recognition,thm:condition_H_automatic} in \cref{section:application_nielsel_realisation}.

\subsection*{Conventions}
This paper is written in the language of $\infty$--categories as set down in \cite{lurieHTT,lurieHA}, and so by a \textit{category} we will always mean an $\infty$--category unless stated otherwise.

\subsection*{Acknowledgements}

We are grateful to Wolfgang
L\"uck and Shmuel Weinberger for numerous helpful conversations and encouragements on
this project. All three authors are supported by the Max Planck Institute for Mathematics in
Bonn. The second and third authors write this article as part of their PhD-thesis. The third author would like to thank the University of Toronto for its hospitality where parts of this article were written. 

\section{Recollections}\label{section:recollections}

There will be two types of equivariance in this paper, each playing a distinct role. The first kind will be defined for an arbitrary Lie group, which is covered in \cref{subsection:equivariance_spaces};  the second kind, covered in \cref{subsection:equivariant_stable_homotopy_theory}, will be defined only for finite groups (in fact, it is defined more generally for compact Lie groups as in \cite{HKK_PD}) and is the one that supports stable homotopy theory and the theory of equivariant Poincar\'e duality in \cref{subsection:equivariant_poincare_duality}. More details on  the materials in \cref{subsection:equivariant_stable_homotopy_theory,subsection:equivariant_poincare_duality} together with  references to the  original sources may be found in \cite{HKK_PD}.

\subsection{Equivariant spaces}\label{subsection:equivariance_spaces}\label{sec:singular_part}
Throughout, let $\Gamma$ be a Lie group.
\begin{nota}
    Let $\orbit(\Gamma)$ be the topological category of homogeneous $\Gamma$--spaces, the full topological subcategory of the category of topological $\Gamma$-spaces
    on objects isomorphic to $\Gamma/H$ for closed subgroups $H \leq \Gamma$.
\end{nota}

\begin{defn}
    The category $\spc_{\Gamma}$ of \textit{ $\Gamma$--spaces }  is defined as the category of presheaves $\presheaf(\orbit(\Gamma))\coloneqq\func(\orbit(\Gamma)\op,\spc)$.
\end{defn}

\begin{cons}[Fundamental adjunctions]\label{cons:fundmanetal_adjunctions}
    Genuine equivariant spaces participate in many adjunctions, the fundamental one that we will need being the following: let $\alpha\colon K\rightarrow \Gamma$ be a continuous homomorphism of Lie groups.
    By left Kan extension and restriction along the (opposite) induction functor $\ind^{\orbit}_{\alpha} \colon \orbit(K)\op \to \orbit(\Gamma)\op$, we obtain the adjunction
    \begin{center}
        \begin{tikzcd}
            \ind_{\alpha}\coloneqq (\ind^{\orbit}_{\alpha})_!\colon \spc_K \ar[r, shift left = 1] & \spc_{\Gamma} \ar[l, shift left = 1] \cocolon \res_{\alpha}\coloneqq (\ind^{\orbit}_{\alpha})^*.
        \end{tikzcd}
    \end{center}
    Specialising to the two cases of  $\alpha=\iota\colon H\rightarrowtail \Gamma$ being a closed subgroup and $\alpha=\theta\colon \Gamma \twoheadrightarrow Q$ being a continuous  surjection of Lie groups with kernel $N$, the adjunction above yields the following two adjunctions which we have given special notations
    \begin{center}
        \begin{tikzcd}
            \ind^{\Gamma}_H\coloneqq \ind_{\iota}\colon \spc_H \ar[r, shift left = 1] & \spc_{\Gamma} \ar[l, shift left = 1] \cocolon \res^{\Gamma}_H\coloneqq \res_{\iota} 
        \end{tikzcd}
    \end{center}
    \begin{center}
        \begin{tikzcd}
            N\backslash(-)\coloneqq \ind_{\theta}\colon \spc_{\Gamma} \ar[r, shift left = 1] & \spc_{Q} \ar[l, shift left = 1] \cocolon \inflated^Q_{\Gamma}\coloneqq \res_{\theta}.  
        \end{tikzcd}
    \end{center}
    Importantly, in the special case of a continuous surjection $\theta\colon \Gamma\twoheadrightarrow Q$, we have an adjunction $\res^{\orbit}_{\theta}\colon \orbit(Q)\op\rightleftharpoons \orbit(\Gamma)\op \cocolon \ind^{\orbit}_{\theta}$, and so $\infl^Q_{\Gamma}=(\ind^{\orbit}_{\theta})^*\simeq (\res^{\orbit}_{\theta})_!$.
\end{cons}

\begin{obs}\label{obs:composition_of_inductions}
    In particular, suppose we have  homomorphisms of Lie groups $\iota\colon N\rightarrowtail \Gamma$ and $\theta\colon \Gamma\twoheadrightarrow Q$ which are injective and surjective, respectively, and such that the composite $\theta\circ \iota\colon N \rightarrow Q$ is also surjective. Writing $\pi$ for the kernel of $\theta$, we thus see that $\ker(\theta\circ\iota)= N\cap \pi$. Since for composable homomorphisms of Lie groups $\alpha$ and $\beta$ we have $\ind_{\alpha\circ \beta}\simeq \ind_{\alpha}\circ\ind_{\beta}$, we see that in this case, there is a natural equivalence of functors $\spc_N\rightarrow \spc_Q$
    \[\pi\backslash\ind^{\Gamma}_N(-)\simeq (N\cap\pi)\backslash(-).\]
\end{obs}

\begin{cons}[Singular parts]\label{cons:singular_part}
    Denote by $s\colon \singorbit(\Gamma) \subseteq \orbit(\Gamma)$ the full subcategory on the orbits $\Gamma/H$ with $H$ nontrivial. We then get the adjunction  
    \begin{center}
        \begin{tikzcd}
            \presheaf(\singorbit(\Gamma)) \rar[shift left = 1.5, "s_!" description] & \presheaf(\orbit(\Gamma)) = \spc_{\Gamma} \lar[shift left = 1.5, "s^*" description]
        \end{tikzcd}
    \end{center}
    by restricting and left Kan extending along $s$. We abbreviate $\singpart{(-)} = s_!s^*$, writing $\varepsilon \colon \singpart{\udl{X}} \rightarrow \udl{X}$ for the adjunction counit. Note that for an orbit $\myuline{\Gamma/H} \in \spc_{\Gamma}$ we have $\singpart{\myuline{\Gamma/H}}\simeq \emptyset$ if $H=e$ and $\varepsilon\colon\singpart{\myuline{\Gamma/H}}\rightarrow \myuline{\Gamma/H}$ is an equivalence otherwise.
    We refer to $\singpart{\udl{X}}$ as the \textit{singular part of $\udl{X}$} and think of  $\varepsilon \colon \singpart{\udl{X}} \rightarrow \udl{X}$ as the inclusion of all points with nontrivial isotropy.
\end{cons}

\begin{nota}
    In the special case of $\Gamma=C_p$, we have $\singorbit(C_p) = \{C_p/C_p\}\simeq \ast$, so that $\presheaf(\singorbit(C_p))\simeq \spc$. In this case, one can work out that $s_!$ just assigns a space to the constant diagram as an object in $\spc_{C_p}$. Moreover, $s^*\udl{X}=X^{C_p}$, and so we will also write $\singpart{\udl{X}}\in\spc_{C_p}$ as $\udl{X^{C_p}}$ interchangeably in this case.
\end{nota}

\begin{fact}
    \label{fact:quotients_and_inflation}
    Let $\Gamma$ and $Q$ be groups, $N \leq \Gamma$ a normal subgroup and $p \colon \Gamma \rightarrow Q$ a surjective group homomorphism. 
    If $Q$ acts on the topological space $X$, then there is a natural $\Gamma/N$-equivariant homeomorphism
    \[N \backslash X \cong p(N) \backslash X.  \]
    Here, $X$ is considered as a $\Gamma$-space via $p$, and $p(N) \backslash X$ is considered as a $\Gamma/N$-space via $\Gamma/N \rightarrow Q/p(N)$. Specialising to orbit categories, we obtain the commutative diagram
    \begin{equation*}
        \begin{tikzcd}
            \orbit(Q) \ar[r, "\res^\orbit"] \ar[d, "\ind^\orbit"] & \orbit(\Gamma) \ar[d, "\ind^\orbit"] \\
        \orbit(Q/p(N)) \ar[r, "\res^\orbit"] & \orbit(\Gamma/N).
        \end{tikzcd}
    \end{equation*}
    Applying $\presheaf(-)$ with the $(-)_!$ functoriality, we get an equivalence of functors
    \[ N \backslash \infl_{\Gamma}^Q(-) \simeq \infl_{\Gamma/N}^{Q/p(N)} p(N) \backslash(-) \colon \spc_{Q} \rightarrow \spc_{\Gamma/N}. \]
\end{fact}

\begin{fact}
    \label{fact:singular_part_and_quotient}
    Given any discrete group $\Gamma$ and a $\Gamma$--space $\udl{X}$, as well as a \textit{proper} normal subgroup $N \subset \Gamma$ such that the $N$-action on $\udl{X}$ is free, the inclusion $\udl{X}^{>1} \rightarrow \udl{X}$ induces equivalences 
    \[  (N \backslash  \udl{X})^{>1} \xleftarrow{\simeq} (N \backslash  \udl{X}^{>1})^{>1} \xrightarrow{\simeq} N \backslash  \udl{X}^{>1}.\]
    Indeed, all functors involved commute with colimits, and the statement is clearly true for orbits, out of which every $\Gamma$--space may be built via colimits.
\end{fact}

\subsection{Genuine equivariant stable homotopy theory}\label{subsection:equivariant_stable_homotopy_theory}
Let $G$ be a finite group. The stable category \textit{ $\spectra_G$ of genuine $G$--spectra } is a refinement of the category $\spectra^{BG}$ of spectra with $G$--action with better formal properties. This is a refinement in that $\spectra^{BG}$ sits fully faithfully in the category $\spectra_G$, in fact in two different ways. One way to define $\spectra_G$, following Barwick, is as the category $\mackey_G(\spectra)\coloneqq \func^{\times}(\spancategory(\finite_G),\spectra)$ of $G$--Mackey functors in spectra. A good introduction to the materials in this subsection may be found, for instance, in \cite[Part 2]{MNN17}.

The category of genuine equivariant spectra is valuable as it is a particularly conducive environment for inductive methods enabled by many compatibility structures between these categories for different groups, expressed in terms of various adjunctions. Moreover, $\spectra_G$ should also be thought of as the ``universal category of equivariant homology theories'' on $\spc_G$. For example, there is a symmetric monoidal colimit--preserving functor ${\Sigma}^{\infty}_{+}\colon \spc_G\rightarrow\spectra_G$ which is the analogue of the suspension spectrum functor nonequivariant spectra, and $\spectra_G$ is generated under colimits by $\{{\Sigma}^{\infty}_+\myuline{G/H}\}_{H\leq G}$.

\begin{fact}[Restriction--(co)induction]
    For a subgroup $H\leq G$, we have the adjunctions
    \begin{center}
        \begin{tikzcd}
            \spectra_G \ar[rr, "\res^G_H" description]&& \spectra_H \ar[ll, "\ind^G_H"description, bend right = 30]\ar[ll, "\coind^G_H"description, bend left = 30]
        \end{tikzcd}
    \end{center}
    where moreover, there is a canonical equivalence of functors $\ind^G_H\simeq \coind^G_H$ classically known as the Wirthm\"uller isomorphism.
\end{fact}

\begin{fact}[Genuine fixed points]\label{fact:genuine_fixed_points}
    There is a functor $(-)^G\colon \spectra_G\rightarrow \spectra$  called the \textit{genuine fixed points  functor} which, from the Mackey functors perspective, is given by evaluating at $G/G\in\finite_G$. This participates in an adjunction 
    \begin{center}
        \begin{tikzcd}
            \inflated^e_G\colon \spectra \ar[rr, shift left = 1] && \spectra_G \ar[ll, shift left = 1] \cocolon (-)^G
        \end{tikzcd}
    \end{center}
    where $\inflated^e_G$  preserves compact objects and is the unique symmetric monoidal colimit preserving functor from $\spectra$ to $\spectra_G$. For every subgroup $H\leq G$, we may also define the genuine $H$--fixed points functor $(-)^H$ as the composite $\spectra_G\xrightarrow{\res^G_H}\spectra_H\xrightarrow{(-)^H}\spectra$.
\end{fact}

\begin{fact}[Borel fixed points]\label{fact:borel_fixed_points}
    There is a standard Bousfield (co)localisation
    \begin{center}
        \begin{tikzcd}
            \spectra_{G}\ar[rr,"\beta^*"description, two heads] && \spectra^{BG}\ar[ll,bend right = 22, "\beta_!"' description, hook]\ar[ll,bend left = 22, "\beta_*"' description, hook]
        \end{tikzcd}
    \end{center}
    where $\beta^*X\simeq X^e$, $\beta_!\beta^*X\simeq EG_+\otimes X$, and $\beta_*\beta^*X\simeq F(EG_+,X)$. This well--known pair of adjunctions may for example be worked out from combining \cite[Thm. 3.9, Prop. 6.5, Prop. 6.6, Prop. 6.17]{MNN17}. Under the Mackey functors perspective, $\beta^*$ is given by evaluating at $G/e\in\finite_G$. In particular, we see that $\spectra^{BG}$ embeds into $\spectra_G$ in two different ways, as mentioned above. Via the functor $\beta^*$ as well as the homotopy orbits $(-)_{hG}$, homotopy fixed points $(-)^{hG}$, and Tate fixed points $(-)^{tG}$ functors $\spectra^{BG}\rightarrow \spectra$, we may also obtain the functors $(-)_{hG}, (-)^{hG}, (-)^{tG}\colon \spectra_G\rightarrow\spectra$ which also fit in a fibre sequence of functors $(-)_{hG}\rightarrow (-)^{hG}\rightarrow (-)^{tG}$. In particular, these functors only depend on the underlying spectrum with $G$--action.
\end{fact}

\begin{fact}[Geometric fixed points]
    There is a symmetric monoidal colimit--preserving functor $\Phi^G(-)\colon \spectra_G\rightarrow \spectra$ called the \textit{geometric fixed points} which is uniquely characterised by sending ${\Sigma}^{\infty}_+\myuline{G/H}$ to $0$ when $H\lneq G$ and to $\sphere$ when $H=G$. For a subgroup $H\leq G$, we may also define $\Phi^H$ as the functor $\spectra_G\xrightarrow{\res^G_H}\spectra_H\xrightarrow{\Phi^H}\spectra$. The collection of functors $\Phi^H\colon \spectra_G\rightarrow \spectra$ for all $H\leq G$ is jointly conservative.
    
    The geometric fixed points functor participates in an adjunction 
    \begin{center}
        \begin{tikzcd}
            \Phi^G\colon \spectra_G \ar[rr, shift left = 1] && \spectra \ar[ll, shift left = 1, hook] \cocolon \Xi^G
        \end{tikzcd}
    \end{center}
    where $\Xi^G$ is fully faithful. For $E\in\spectra$, $\Xi^GE\in\spectra_G$ is concretely given by the $G$--Mackey functor which assigns $E$ to $G/G$ and $0$ to all $G/H$ for $H\lneq G$. 

    Furthermore, using that $\spectra$ is the initial presentably symmetric monoidal stable category, it is also not hard to see that $\Phi^G\inflated^e_G\simeq \id_{\spectra}$.
\end{fact}

Next, we recall the standard decomposition in the special case of genuine $C_p$--spectra, which is all that we will need in our work.

\begin{fact}[$C_p$--stable recollement]\label{recollect:borelification_(co)localisation}
    Let $G=C_p$. In this case, some of the adjunctions we have seen fit into a stable recollement (also called split Verdier sequence)
    \begin{center}
        \begin{tikzcd}
            \spectra\ar[rr,"\Xi^{C_p}"description,hook] &&\spectra_{C_p}\ar[rr,"\beta^*"description, two heads] \ar[ll,bend right = 25, "\Phi^{C_p}"' description, two heads]\ar[ll,bend left = 25, two heads] && \spectra^{BC_p}\ar[ll,bend right = 22, "\beta_!"' description, hook]\ar[ll,bend left = 22, "\beta_*"' description, hook]
        \end{tikzcd}
    \end{center}
    in that the top two layers of composites are fibre--cofibre sequences of presentable stable categories. This may be deduced, for example, from a combination of \cite[\textsection 6.4]{MNN17} and \cite[\textsection A.2]{Hermitian2}. From this, one obtains for every $E\in\spectra_{C_p}$ a pullback square
    \begin{center}
        \begin{tikzcd}
            E^{C_p}\rar\dar \ar[dr, phantom , "\lrcorner", near start]& \Phi^{C_p}E\dar\\
            E^{hC_p}\rar & E^{tC_p}
        \end{tikzcd}
    \end{center}
    of spectra (c.f. for instance \cite[Thm. 6.24]{MNN17} or \cite[Prop. A.2.12]{Hermitian2}).
\end{fact}

\subsection{Equivariant Poincar\'e duality}\label{subsection:equivariant_poincare_duality}

Let $G$ be a finite group.
We briefly recall the theory of $G$-equivariant Poincar\'e duality spaces, which is built upon the notion of $G$--categories. Recall that the category $\cat_G$ of $G$--categories is defined as $\func(\orbit(G)\op,\cat)$, akin to the category of $G$--spaces. This category admits an internal functor category $\udl{\func}(\udlcatC,\udlcatD)$ for each pair $\udlcatC, \udlcatD\in\cat_G$. This satisfies \[\udl{\func}(\udlcatC,\udlcatD) (G/H) \simeq \func_H(\res^G_H\udlcatC,\res^G_H\udlcatD),\] where the latter is the category of $H$--functors from $\res^G_H\udlcatC$ to $\res^G_H\udlcatD$.  A very important $G$-category for us is the $G$-category $\myuline{\spectra}$ of genuine $G$-spectra given by  $\myuline{\spectra}(G/H) = \spectra_H$.

Since $\spc_G$ is a full subcategory of $\cat_G$, we may view a $G$--space $\udl{X}$ as an object in $\cat_G$
For a $G$-space $\udl{X}$ we denote the unique map to the point by
\[ \uniquemap{X} \colon \udl{X} \rightarrow \udl{*}. \]
Write $\myuline{\spectra}^{\udl{X}} = \udl{\func}(\udl{X},\myuline{\spectra})$ for the category of equivariant local systems on $\udl{X}$. Explicitly, that amounts to specifying a local system of $H$-spectra $X^H \rightarrow \spectra_H$ for each subgroup $H \subset G$ plus compatibilities.
Colimit, restriction and limit of local systems give two adjunctions
\begin{equation*}
\begin{tikzcd}
    \myuline{\spectra}^{\udl{X}} \arrow[rr, "\uniquemap{X}_!" description, bend left] \arrow[rr, "\uniquemap{X}_*"' description, bend right] 
    &  
    & \myuline{\spectra}. \arrow[ll, "\uniquemap{X}^*"'description]
\end{tikzcd}
\end{equation*}
One should think of the colimit $X_! E$ of an equivariant local system $E$ on $\udl{X}$ as equivariant homology of $X$ twisted by $E$ and similarly of the limit as equivariant twisted cohomology.

The following is a recollection from \cite[Sec. 4.1.]{HKK_PD}.
A compact $G$-space $\udl{X}$ admits an equivariant dualising spectrum $D_{\udl{X}} \in \func_G(\udl{X}, \myuline{\spectra})$ which comes together with a collapse map $c \colon \sphere_G \rightarrow \uniquemap{X}_! D_{\udl{X}}$.
These are uniquely characterised by the property that the induced capping map
\begin{equation}\label{eq:capping_equivalence}
    \ambi{c}{\xi}{(-)} \colon \uniquemap{X}_*(-) \to \uniquemap{X}_!(D_{\udl{X}} \otimes -)
\end{equation}
is an equivalence.
Applying fixed points and homotopy groups, the collapse map really corresponds to a class in twisted equivariant homology such that capping with it induces an equivalence between equivariant cohomology and twisted equivariant homology.
Let us just mention that there is the larger class of \textit{twisted ambidextrous} $G$-spaces for which an equivalence of the form \cref{eq:capping_equivalence} exists. 

\begin{defn}
    A compact (or twisted ambidextrous) $G$-space $\udl{X}$ is called \textit{$G$-Poincar\'e} if the dualising spectrum $D_{\udl{X}}$ takes values in $\udl{\Pic}(\myuline{\spectra})$.
\end{defn}

\begin{nota}
    Let $\xi \in \func_G(\udl{X},\myuline{\spectra})$ be a local system of $G$-spectra on the $G$-space $\udl{X}$. For an $H$-fixed point $y \in X^H$, i.e. a map $y \colon \myuline{G/H} \rightarrow \udl{X}$, using the composition
    \[ \func_G(\udl{X},\myuline{\spectra}) \xrightarrow{y^*} \func_G(\myuline{G/H},\myuline{\spectra}) \simeq \spectra_H \]
    we obtain an $H$-spectrum that we will denote by $\xi(y)$.
\end{nota}

Note that a compact (or twisted ambidextrous) $G$-space $\udl{X}$ is $G$-Poincar\'e if and only if for each $y \in X^H$ the value $D_{\udl{X}}(y) \in \spectra_H$ is an invertible $H$-spectrum.

\begin{thm}[{\cite[Thm 4.2.9.]{HKK_PD}}]\label{thm:HKK_Thm.4.2.9}
    Let $\udl{X}$ be a $G$-Poincar\'e space. Then for each closed subgroup $H \leq G$, the space $X^H$ is a (nonequivariant) Poincar\'e space. Moreover, its dualising spectrum is given as the composite 
    \begin{equation*}
        X^H \xrightarrow{D_{\udl{X}}} \spectra_H \xrightarrow{\Phi^H} \spectra.
    \end{equation*}
\end{thm}

\begin{example}[ {\cite[Cor. 5.1.16.]{HKK_PD}}]\label{example:jones_counterexample}
    Let $p$ be an odd prime and $k \geq 1$ an integer.
    There exists a compact $C_p$-space $\udl{X}$ for which
    $X^e$ is contractible while $X^{C_p} \simeq \mathbb{R} P^{2k}$. None such $C_p$-space is $C_p$-Poincar\'e.
    In particular, there are compact $G$-spaces such that all fixed points are nonequivariant Poincar\'e spaces which are not themselves $G$-Poincar\'e.
\end{example}

\section{Poincar\'e duality for the group \texorpdfstring{$C_p$}{Cp}}\label{section:PD_for_C_p}

In this section, we investigate equivariant Poincar\'e duality for the group $C_p$ more closely. 
Our goal is to prove \cref{mainthm:recognition} (c.f. \cref{thm:characterization_of_Cp_PD}) which gives a somewhat computable condition on a \textit{compact} $C_p$-space $\udl{X}$ to be $C_p$-Poincar\'e assuming that $X^e$ and $X^{C_p}$ are nonequivariant Poincar\'e spaces. This amounts to checking that $D_{\udl{X}}\colon \udl{X}\rightarrow \myuline{\spectra}$ lands in invertible objects. 

Since invertibility is a pointwise condition and since we already know that $D_{X^e}\colon X^e\rightarrow \spectra$ lands in invertible objects, it suffices to show that $D_{\udl{X}}(x)\in\spectra_{C_p}$ is invertible for every $x\in X^{C_p}$. Moreover, from our hypothesis and \cref{thm:HKK_Thm.4.2.9}, we already know that $D_{\udl{X}}(x)^e$ and $\Phi^{C_p}D_{\udl{X}}(x)$ are invertible  spectra. This consideration leads us to record the following well--known observation.
\begin{lem}\label{lem:invertible_compact_relation}
    Let $E \in \spectra_{C_p}$ be such that $E^e$ and $\Phi^{C_p} E$ are invertible. Then:
    \[ \text{$E$ is invertible} \hspace{3mm} \iff \hspace{3mm} \text{$E$ is dualisable} \hspace{3mm} \iff \hspace{3mm} \text{$E$ is compact}. \]
\end{lem}
\begin{proof}
    In $\spectra_{C_p}$, dualisablity and compactness are equivalent, and invertible spectra are dualisable. 
    So it suffices to show that if $E$ is dualisable, then it is invertible, i.e. that the counit $E \otimes E^{\vee} \rightarrow \sphere_{C_p}$
    is an equivalence. 
    But this can be checked after applying $(-)^e$ and $\Phi^{C_p}(-)$, which are jointly conservative.
    As both of these functors are symmetric monoidal, the counit for $E$ is sent to the counit for $E^e$ and $\Phi^{C_p} E$, both of which we assumed to be equivalences.
\end{proof}

Thus, by virtue of \cref{lem:invertible_compact_relation}, our task at hand is tantamount to ensuring that the $C_p$-spectrum $D_{\udl{X}}(x)$ is compact for every $x\in X^{C_p}$. To this end, we will employ various cellular manoeuvres in \cref{sec:cellular_manouvres} to obtain ``compact approximations'' to any $C_p$--spectrum; we then characterise compactness of a $C_p$--spectrum with vanishing geometric fixed points through its underlying Borel-$C_p$-spectrum in \cref{sec:compact_cp_spectra_characterisation}. Lastly, we combine all these in  \cref{sec:Cp_PD_recognition} to obtain the promised recognition principle for $C_p$-Poincar\'e spaces.

\begin{rmk}
    Our work on $C_p$-spectra heavily drew inspiration from at least two sources. The first one being \cite{krause2020picard}, which gives a nice computation of the Picard group of $\spectra_{C_p}$, and whose methods we expand on. The second one is \cite{shahDualisables}, which gives another compactness (or dualisability) criterion for $C_p$-spectra. Our approach is not exactly tailored to the methods in the latter source, and we will not need to refer to them, but it might be possible that they give another way of proving the main results in this section.
\end{rmk}

\subsection{Cellular manoeuvres and compact approximations}\label{sec:cellular_manouvres}

Recall that a $C_p$-spectrum is \textit{finite} if it lies in the stable subcategory of $\spectra_{C_p}$ generated by $\Sigma^\infty_+ \myuline{C_p/e} = \ind_e^{C_p} \sphere$ and $\Sigma^\infty_+ \myuline{C_p/C_p} = \infl_e^{C_p} \sphere$. 
A $C_p$-spectrum is compact if and only if it is a retract of a finite $C_p$-spectrum.

\begin{lem}[Tucking the free part]
    \label{lem:tucking}
    Let $X \in \spectra_{C_p}$ be such that $X^e$ is bounded below and such that $\pi_k(X^e)$ is a finitely generated abelian group for $k \leq N$ for some $N$. Then there is a fiber sequence
    \[ F \rightarrow X \rightarrow Y \]
    in $\spectra_{C_p}$ such that $F$ is finite, $Y^e$ is $N$-connected and $\Phi^{C_p} X \rightarrow \Phi^{C_p} Y$ is an equivalence.
\end{lem}

\begin{proof}
    Note that if $A \rightarrow B$ and $B \rightarrow C$ are maps of $C_p$-spectra whose fibers are compact with trivial geometric fixed points, then the composition $A \rightarrow C$ satisfies the same condition. Thus, by induction it suffices to consider the case where $N=0$ and $X^e$ is $-1$-connected. Pick a finite set of generators $\{ f_i \colon \sphere \rightarrow  X^e \}$ of $\pi_0(X^e)$. Then the composition
    \[   f = \left( \bigoplus_i \ind \sphere \xrightarrow{\bigoplus_i \ind f_i} \ind X^e \xrightarrow{c} X \right), \]
    where $c$ denotes the counit, induces a surjection on $\pi_0$ upon applying $(-)^e$. Now define $Y$ to be the cofiber of $f$ and $F$ its source. This finishes the proof since $\ind \sphere$ is compact and satisfies $\Phi^{C_p} \ind \sphere \simeq 0$.
\end{proof}

\begin{rmk}
    Unlike taking the appropriate connective covers, the procedure of tucking cannot be used in general to kill the homotopy groups of $X^e$. The reason is that the effect on the next higher homotopy group is quite brutal. However, if $X^{e}$ is $(l-1)$-connected and $\pi_l X^e$ is a finitely generated free $\mathbb{Z}[C_p]$-module, then the proof of \cref{lem:tucking} shows that we can kill $\pi_l X^e$ while making sure that the next higher homology group is unchanged. On the other hand, tucking preserves the geometric fixed points whereas the aforementioned connective covers do not. 
\end{rmk}

\begin{lem}
    \label{lem:extending_maps_from_geom_fixedpoints}
    Let $Q$ be a compact spectrum, $E$ a $C_p$-spectrum and $f \colon Q \rightarrow \Phi^{C_p}E$ a map. 
    Then there exists a compact $C_p$-spectrum $F$, a map $g \colon F \rightarrow E$, and an identification $\Phi^{C_p} F \simeq Q$ under which $\Phi^{C_p}g = f$.
\end{lem}
\begin{proof}
    First we reduce to the case where $E$ is compact.
    Write $E = \colim_{i \in I} E_i$ as a filtered colimit of compact $C_p$--spectra. 
    As $\Phi^{C_p}$ commutes with colimits and $Q$ is compact, there is some $i \in I$ for which $f \colon Q \to \Phi^{C_p} E$ factors through the compact spectrum $\Phi^{C_p} E_i$.
    If now $F$ is a compact $C_p$-spectrum together with a map $F \to E_i$ which induces the map $Q \to \Phi^{C_p} E_i$ on geometric fixed points, then the composite $F \to E_i \to E$ satisfies the claim.

    Now assume that $E$ is compact. 
    Then there exists $k \in \bbZ$ such that each map $Q \rightarrow T$ to a $k$-connected spectrum $T$ is nullhomotopic.
    By \cref{lem:tucking}, we can find $U \in \spectra_{C_p}$ together with a map $E \rightarrow U$ that has compact fiber and that induces an equivalence on geometric fixed points, such that $U^e$ is $(k-1)$-connected. Thus, $\Sigma U^e$ is $k$-connected, and consequently also $\Sigma U_{hC_p}$ is $k$-connected. Note that $U$ is compact as well.

    Consider the following diagram where the  lower horizontal maps form a fiber sequence and  the square is cartesian by \cref{fact:borel_fixed_points,recollect:borelification_(co)localisation}. 
    \begin{center}
    \begin{tikzcd}
        Q \arrow[rr] \arrow[rrrdd, "\simeq 0", bend left=70] \arrow[rdd, "a", dashed, bend right] \arrow[rd, "b", dashed] 
        &                             
        & \Phi^{C_p} E \arrow[d] 
        &                 
        \\
        & U^{C_p} \ar[dr, "\lrcorner", phantom, near start] \arrow[r] \arrow[d] 
        & \Phi^{C_p} U \arrow[d] 
        &                 
        \\
        & U^{hC_p} \arrow[r]          
        & U^{tC_p} \arrow[r]     
        & \Sigma U_{hC_p}
    \end{tikzcd}
    \end{center}
    The nullhomotopy of the long bent arrow induces the dashed morphism $a$, which in turn determines the dashed morphism $b$. By the adjunction from \cref{fact:genuine_fixed_points}, the map $b$ is adjoint to a map
    $q \colon \infl_{C_p}^e Q \rightarrow U$ which induces the map $Q \rightarrow \Phi^{C_p} U$ on geometric fixed points.  This fits into a fibre sequence $F \rightarrow E \rightarrow \cofib(q)$ where the map $E \to \cofib(q)$ is induced by the map $E \to U$ from above.
    Note that $F$ is compact as $E$ is compact, $U$ is compact as observed above, and $\infl_{C_p}^e$ preserves compactness.
    On geometric fixed points, under the identification $\Phi^{C_p} E \xrightarrow{\simeq} \Phi^{C_p} U$, this gives
    \[ \Phi^{C_p} F \simeq \fib(\Phi^{C_p} E \rightarrow \cofib(Q \rightarrow \Phi^{C_p} E)) \simeq Q \]
    as desired.
    Clearly, this identifies the map $\Phi^{C_p} F \rightarrow \Phi^{C_p}E$ with $f$.
\end{proof}

\begin{cor}\label{cor:making_almost_free}
    If $E$ is a $C_p$-spectrum with $\Phi^{C_p} E$ compact, then there exists a finite $C_p$-spectrum $F$ and a map $g \colon F \rightarrow E$
    which induces an equivalence on geometric fixed points.
\end{cor}

\begin{proof}
    Set $f = \id_{\Phi^{C_p}E}$ in \cref{lem:extending_maps_from_geom_fixedpoints}.
\end{proof}

The following lemma will be useful later.

\begin{lem}
    \label{lem:technical_comparison}
    Consider a cospan $X \xrightarrow{f} Z \xleftarrow{g} Y$ in $\spectra_{C_p}$  where $f$ and $g$ induce equivalences on geometric fixed points. Additionally suppose $\Phi^{C_p} X\in\spectra$ is compact. Then there exists a commutative square
    \begin{center}
    \begin{tikzcd}
        F \ar[r] \ar[d]   & Y \arrow[d, "g"] \\
        X \arrow[r, "f"] & Z          
    \end{tikzcd}
    \end{center} 
    with $F$ compact such that all maps are equivalences on geometric fixed points.
\end{lem}

\begin{proof}
    Let $E = X \times_Z Y$ and find a fiber sequence $F \rightarrow E \rightarrow E'$
    with $F$ compact and $\Phi^{C_p} E' \simeq 0$ as provided by \cref{cor:making_almost_free}.
    Then the outer quadrilateral in the diagram
    \begin{center}
    \begin{tikzcd}
        F \arrow[rdd, bend right] \arrow[rd] \arrow[rrd, bend left] 
        &                       
        &             
        \\
        & E \arrow[d] \arrow[r] 
        & Y \arrow[d] 
        \\
        & X \arrow[r]          
        & Z          
    \end{tikzcd}
    \end{center}
    has the desired properties.
\end{proof}

\subsection{Compact and induced \texorpdfstring{$C_p$}--spectra}\label{sec:compact_cp_spectra_characterisation}

In this section we characterise compactness for $C_p$-spectra with trivial geometric fixed points through their underlying Borel $C_p$-spectrum.
Notice that $C_p$--spectra with vanishing geometric fixed points have the following crucial properties.

\begin{lem}[CoBorel compactness]\label{lem:coborel_compactness}
    Let $X$ be a $C_p$-spectrum with $\Phi^{C_p} X \simeq 0$. Then
    \begin{myenum}{(\arabic*)}
        \item for every $Y\in \spectra_{C_p}$ the map $(-)^e\colon \map_{\spectra_{C_p}}(X,Y) \rightarrow \map_{\spectra^{BC_p}}(X^e,Y^e)$ is an equivalence;
        \item the $C_p$-spectrum $X$ is compact in $\spectra_{C_p}$ if and only if $X^e$ is compact in $\spectra^{BC_p}$.
    \end{myenum}
\end{lem}
\begin{proof}
    We use the notations from \cref{recollect:borelification_(co)localisation}. Observe by \cref{recollect:borelification_(co)localisation} that $\Phi^{C_p}X\simeq 0$ is equivalent to the condition that the adjunction counit $\beta_!\beta^*X\rightarrow X$ is an equivalence. Now for (1), just note that
    $\map_{\spectra_{C_p}}(X,Y) \xleftarrow{\simeq}\map_{\spectra_{C_p}}(\beta_!\beta^*X,Y) \xrightarrow{\simeq}\map_{\spectra^{BC_p}}(\beta^*X,\beta^*Y)$ as claimed.    For (2), note that $\beta_! \colon \spectra^{BC_p}\rightarrow \spectra_{C_p}$ preserves and detects compactness as it is fully faithful and admits the colimit preserving right adjoint $\beta^*$.
    This shows that $X^e = \beta^*X \in\spectra^{BC_p}$ is compact if and only if $X \simeq \beta_! \beta^* X \in \spectra_{C_p}$ is compact.
\end{proof}

\begin{nota}\label{nota:induced_spectra}
    Following the notation from \cite[Def. I.3.7]{nikolausScholze}, we write $\spectra^{BC_p}_{\mathrm{ind}}\subseteq \spectra^{BC_p}$ for the smallest idempotent--complete stable subcategory generated by the image of the functor $\ind^{C_p}_e\colon \spectra\rightarrow \spectra^{BC_p}$.  Similarly, we write $\spectra^{\mathrm{ind}}_{C_p}\subseteq \spectra_{C_p}$ for the smallest idempotent--complete stable subcategory containing the image of $\ind^{C_p}_e\colon \spectra\rightarrow\spectra_{C_p}$. 
    By \cref{lem:coborel_compactness} (1), the functor $(\--)^e \colon \spectra_{C_p} \to \spectra^{BC_p}$ restricts to a fully faithful functor  $\spectra^{\mathrm{ind}}_{C_p} \rightarrow \spectra^{BC_p}_{\mathrm{ind}}$, which is also essentially surjective (and so is an equivalence) since $(-)^e=\beta^*$ is essentially surjective and $\beta_!$ and $\beta^*$ are compatible with $\ind^{C_p}_e$.
\end{nota}

\begin{lem}\label{lem:induced_spectra_with_underlying_compact}
   As full subcategories of $\spectra^{BC_p}$, we have the equality
   \[ (\spectra^{\omega})^{BC_p}\cap\spectra^{BC_p}_{\mathrm{ind}} = (\spectra^{BC_p})^{\omega}. \]
   Thus, if $E \in \spectra_{C_p}$ with $\Phi^{C_p} E = 0$, then $E$ is compact if and only if $E^e \in (\spectra^{\omega})^{BC_p}\cap\spectra^{BC_p}_{\mathrm{ind}}$.
\end{lem}
\begin{proof}
    The inclusion $(\spectra^{\omega})^{BC_p}\cap\spectra^{BC_p}_{\mathrm{ind}} \supseteq (\spectra^{BC_p})^{\omega} $ is clear since $(\spectra^{BC_p})^{\omega}$ is generated under finite colimits and retracts by $\susps_+C_p/e\simeq \ind^{C_p}_e\sphere$. For the converse, we  use that 
    \[(\spectra^{BC_p})^{\omega} \hookrightarrow (\spectra^{\omega})^{BC_p} \twoheadrightarrow \stmodSmall_{C_p}(\sphere)\coloneqq (\spectra^{\omega})^{BC_p}/(\spectra^{BC_p})^{\omega}\]
    is a fibre sequence of small stable categories (c.f. for instance \cite[Lem. A.1.8]{Hermitian1}) and that for $X,Y\in(\spectra^{\omega})^{BC_p}$ we have the formula (c.f. for instance \cite[Lem. 4.2]{krause2020picard})
    \[\mapsp_{\stmodSmall_{C_p}(\sphere)}(X,Y)\simeq (Y\otimes DX)^{tC_p}, \] where $DX$ is the pointwise Spanier--Whitehead dual in $(\spectra^{\omega})^{BC_p}$. 

    Observe that for any $X\in\spectra^{BC_p}$ and $Y\in\spectra^{BC_p}_{\mathrm{ind}}$ one has $\big(Y\otimes \mapsp(X,\sphere)\big)^{tC_p}\simeq 0$ owing to the fact that $(-)^{tC_p}$ vanishes on $\spectra^{BC_p}_{\mathrm{ind}}$ and that $\spectra^{BC_p}_{\mathrm{ind}}\subseteq \spectra^{BC_p}$ is a tensor--ideal. Therefore, for $Z\in (\spectra^{\omega})^{BC_p}\cap\spectra^{BC_p}_{\mathrm{ind}}$, we see that
    \[\mapsp_{\stmodSmall_{C_p}(\sphere)}(Z,Z)\simeq (Z\otimes DZ)^{tC_p}\simeq 0\]
    and so $Z$ is in the kernel of the functor $(\spectra^{\omega})^{BC_p}\rightarrow \stmodSmall_{C_p}(\sphere)$. Hence, by the fibre sequence above, we see that $Z\in (\spectra^{BC_p})^{\omega}$ as required.

    The statement about compact $C_p$-spectra follows by combining the first part with \cref{lem:coborel_compactness} (2).
\end{proof}

\subsection{Recognising \texorpdfstring{$C_p$}--Poincar\'e spaces}\label{sec:Cp_PD_recognition}

\begin{cons}[Contravariant functoriality of dualising spectra]
    Consider a map $f \colon \udl{Y} \to \udl{X}$ in $\spc_{C_p}^{\omega}$. 
    We explain how to construct a canonical ``wrong--way'' map 
    \begin{equation}\label{eqn:canonical_wrong_way_map_to_singular_part}
        \beckChevalley^f \colon D_{\underline{X}} \longrightarrow f_!D_{\underline{Y}}.
    \end{equation}
    Combining the contravariant functoriality of cohomology from \cite[Construction 3.4.1]{HKK_PD} with the defining property of the dualising spectrum, we obtain the natural transformation
    \begin{equation*}
        X_! (D_{\udl{X}} \otimes -) 
        \simeq X_*(-) 
        \xrightarrow{\beckChevalley^f_*} Y_* f^*(\--) 
        \simeq Y_!(D_{{\udl{Y}}} \otimes f^*(\--)) 
        \simeq X_!(f_! D_{{\udl{Y}}} \otimes -)
    \end{equation*}
    By the classification of colimit preserving functors, see \cite[Corollary 2.30]{Cnossen2023} or \cite[Theorem 2.1.37]{HKK_PD}, this is induced by a map $\beckChevalley^f \colon D_{\udl{X}} \to f_! D_{{\udl{Y}}}$.

    Now consider $\underline{X}\in\spc_{C_p}^{\omega}$.
    For the inclusion $\varepsilon\colon \underline{X^{C_p}}\rightarrow \underline{X}$ of the singular part from \cref{cons:singular_part} we thus obtain a map $\beckChevalley^\varepsilon \colon D_{\underline{X}}\rightarrow \varepsilon_!D_{\underline{X^{C_p}}}$. 
    Applying $\varepsilon^*$ yields the map
\begin{equation}\label{eqn:pulled_back_canonical_wrong-way}
    \varepsilon^* \beckChevalley^\varepsilon \colon \varepsilon^*D_{\underline{X}} \longrightarrow \varepsilon^*\varepsilon_!D_{\underline{X^{C_p}}},
\end{equation}
which may be viewed as a morphism in the nonparametrised functor category $\func(X^{C_p},\spectra_{C_p}) \simeq \func_{C_p}(\udl{X^{C_p}}, \myuline{\spectra})$ - this equivalence may be obtained by applying \cite[Lem. 2.1.16]{HKK_PD} to the adjunction $\infl^e_G\colon \spc\rightleftharpoons \spc_G \cocolon (-)^G$.
\end{cons}

The wrong--way map \cref{eqn:pulled_back_canonical_wrong-way} satisfies the following key vanishing result permitting our characterisation of $C_p$--Poincar\'e spaces. By virtue of the lemma, the cofibre of \cref{eqn:pulled_back_canonical_wrong-way} may be viewed as measuring the ``geometric free part'' of the dualising sheaf $D_{\udl{X}}$. 

\begin{lem}\label{lem:equivalence_mod_induced}
    Let $\underline{X}\in\spc_{C_p}^{\omega}$ and let $\nu \colon \spectra_{C_p} \rightarrow \category{D}$ be an exact functor which vanishes on $\spectra^{\mathrm{ind}}_{C_p}$. Then the map
    \[ \nu(\varepsilon^* D_{\underline{X}}) \longrightarrow \nu(\varepsilon^* \varepsilon_! D_{\underline{X^{C_p}}}) \]
    in $\func(X^{C_p},\category{D})$ induced by \cref{eqn:pulled_back_canonical_wrong-way} is an equivalence. 
\end{lem}

\begin{proof}
    We have to show,  for any $x \in X^{C_p}$, that the map
    $\nu(D_{\udl{X}}(x)) \xrightarrow{\simeq} \nu(f_! D_{\udl{X^{C_p}}}(x))$ is an equivalence.
    First, let us show that for any compact $C_p$-space $\udl{X}$ the map $\varepsilon \colon \udl{X^{C_p}} \rightarrow \udl{X}$ induces an equivalence
    \[ \nu(X_*(-)) \xlongrightarrow{\simeq} \nu(X^{C_p}_*\varepsilon^*(-)). \]
    If $\varepsilon \colon \udl{X^{C_p}} \rightarrow \udl{X}$ is an equivalence, this is a tautology. The class of spaces for which the assertion is true is moreover stable under pushouts, retracts and contains $\udl{Y} = \udl{C_p/e}$, as 
    \[ 0 \simeq \nu(\emptyset_*\epsilon^*(-)) \simeq \nu(\udl{C_p/e}_*(-))\simeq \nu(\udl{C_p/e}_!(-)) \simeq \nu(\ind_e^{C_p}(-)) \simeq 0. \]
    Using that $X_* \simeq X_!(D_{\udl{X}} \otimes -)$ and $X^{C_p}_* \epsilon^* \simeq X^{C_p}_! (D_{\udl{X^{C_p}}} \otimes \epsilon^* -) \simeq  X_! (\epsilon_! D_{\udl{X^{C_p}}} \otimes -)$ we obtain the equivalence
    \begin{equation*}
        \overline{\nu}(X_! (D_{\udl{X}} \otimes \--)) \xlongrightarrow{\simeq} \overline{\nu}(X_!(\varepsilon_! D_{\udl{X^{C_p}}} \otimes \--)).
    \end{equation*}

    Now consider a fixed point $x \colon * \rightarrow X^{C_p}$ (which we also view as $x\colon \terminalTCat \rightarrow \underline{X}$).
    Note that the projection formula provides an equivalence, natural in $E\in\myuline{\spectra}^{\underline{X}}$
    \[ X_!(E \otimes x_!(\sphere_{C_p}))\simeq X_! x_!(x^*E) \simeq x^*E = E(x). \]
    Thus, for any $x \in X^{C_p}$, the map
    $\nu(D_{\udl{X}}(x)) \xlongrightarrow{\simeq} \nu(\epsilon_! D_{\udl{X^{C_p}}}(x))$ is an equivalence, whence the result.
\end{proof}

We are now ready to prove our main result characterising $C_p$--Poincar\'e spaces. Note that, unlike \cref{lem:equivalence_mod_induced}, the key characterising property is given solely in terms of $D_{\udl{X^{C_p}}}$ and does not involve $D_{\udl{X}}$.

\begin{thm}
    \label{thm:characterization_of_Cp_PD}
    Let $\udl{X}$ be a compact $C_p$-space for which $X^e$ and $X^{C_p}$ are (nonequivariant) Poincaré spaces. Then $\udl{X}$ is $C_p$-Poincaré if and only if the cofiber
    \[ \cofib (D_{\udl{X^{C_p}}} \rightarrow \varepsilon^* \varepsilon_{!} D_{\udl{X^{C_p}}})^e \in \func(X^{C_p}, \spectra^{BC_p}) \]
    pointwise lies in the stable subcategory $\spectra^{BC_p}_{\mathrm{ind}}\subseteq \spectra^{BC_p}$.
\end{thm}

\begin{proof}
    As $X^e$ is assumed to be Poincar\'e the specta $D_{\udl{X}}(y) = D_{X^e}(y)\in\spectra$ are invertible for all $y\in X^e$.
    Furthermore, as $X^{C_p}$ is Poincar\'e we know that  $\Phi^{C_p} D_{\udl{X}}(x) = D_{X^{C_p}}(x)$ is invertible for all $x\in X^{C_p}$.
    It now follows from  \cref{lem:invertible_compact_relation} that $\udl{X}$ is $C_p$-Poincar\'e if and only if the $C_p$-spectrum $D_{\udl{X}}(x)$ is compact for all points $x\in X^{C_p}$.

    Note that all maps in the bottom right cospan in the diagram
    \begin{equation}\label{diag:square_cpt_dualising_spectra}
    \begin{tikzcd}
        F \ar[d, dashed, "g'"] \ar[r, dashed, "f'"]
        & D_{\udl{X^{C_p}}}(x) \arrow[d, "g"]
        \\
        D_{\udl{X}}(x) \arrow[r, "f"]
        & \varepsilon_{!} D_{\udl{X^{C_p}}}(x)
    \end{tikzcd}
    \end{equation}
    induce equivalences on geometric fixed points: the map $f$ by  \cref{lem:equivalence_mod_induced}, and the map $g$ by \cite[Lem. 4.2.3]{HKK_PD}.    We can use \cref{lem:technical_comparison} to complete \cref{diag:square_cpt_dualising_spectra} to a commutative square of $C_p$-spectra where $F$ is compact and all maps are equivalences on geometric fixed points.
    Consider the exact functor
    \begin{equation*}
        \nu = \left(\spectra_{C_p} \xrightarrow{(-)^e} \spectra^{BC_p} \rightarrow \spectra^{BC_p}/\spectra^{BC_p}_{\induced} \right).
    \end{equation*}
    Note that as $g'$ and $f'$ are maps between compact $C_p$-spectra that induce an equivalence on geometric fixed points, \cref{lem:induced_spectra_with_underlying_compact} shows that $\nu(\cofib(f')) \simeq \nu(\cofib(g')) \simeq 0$, so $\nu(f')$ and $\nu(g')$ are equivalences.

    Let us first assume that $\udl{X}$ is $C_p$-Poincar\'e.
    It follows from \cref{lem:equivalence_mod_induced} that $\nu(\cofib(f)) \simeq 0$, so also $\nu(f)$ is an equivalence.
    Thus, $\nu(g)$ is an equivalence from which we obtain $\nu(\cofib(g)) \simeq 0$, proving one direction of the claim.

    For the other direction, assume $\nu(\cofib(g)) \simeq 0$, i.e. that $\nu(g)$ is an equivalence.
    As before, as $\nu(f)$ and $\nu(f')$ are equivalences we obtain that $\nu(\cofib(g')) \simeq 0$.
    By definition of $\nu$, this means $\cofib(g')^e \in \spectra^{BC_p}_{\mathrm{ind}}$. 
    Now, $F\in\spectra_{C_p}$ and $D_{\underline{X}}(x)\in\spectra_{C_p}$ both have compact underlying spectra.  Hence, it follows from \cref{lem:induced_spectra_with_underlying_compact}  that $\cofib(g')\in\spectra_{C_p}$ is compact.
    But then $D_{\underline{X}}(x)$ is compact too, as was to be shown.
\end{proof}

\section{Genuine virtual Poincar\'e duality groups}\label{section:gvpd_groups}

In this section, we define a refinement of the classical notion of  \textit{virtual Poincar\'e duality groups}. Recall that a Poincar\'e duality group is a discrete group $\pi$ such that $B\pi$ is a Poincar\'e space, and a group is a virtual Poincar\'e duality group if it contains a Poincar\'e duality group of finite index. In this case, every finite index torsionfree subgroup will be a Poincar\'e duality group.

Now, if $\Gamma$ is a discrete group and $\pi$ a finite index torsionfree normal subgroup, then the space $B\pi$ can be enhanced to a $\Gamma/\pi$-space in a canonical way by viewing it as the quotient $\pi \backslash \classifyingspace{\finite}{\Gamma}$ of the universal space for the family of finite subgroups of $\Gamma$.
One might naturally wonder if that $\Gamma/\pi$-space is $\Gamma/\pi$-Poincar\'e, in which case we call $\Gamma$ a \textit{genuine virtual Poincar\'e duality group}. 
In fact, we will give a slightly more general definition that also includes the case where $\Gamma$ is a Lie group. 
Heuristically, genuine virtual Poincar\'e duality groups are those  which capture the homotopical properties of  groups for which the universal space of proper actions admits a smooth manifold model. 

\subsection{Universal spaces for proper actions}
Let us first collect some examples and constructions for universal spaces for proper actions from the literature.

\begin{defn}
    Let $\Gamma$ be a Lie group. By $\cptfamily$
    we denote the \textit{family of compact subgroups of $\Gamma$}. 
    The \textit{universal space for proper actions} is the universal space for the family $\cptfamily$, and is denoted by $\classifyingspace{\cptfamily}{\Gamma}$.
\end{defn}

If $\Gamma$ is discrete, the family of compact subgroups of $\Gamma$ agrees with the family of finite subgroups, and we denote it by $\finite$. 

\begin{example}[\cite{Abels_universal}, Thm 4.15.]
    \label{ex:manifold_models_from_nonpositively_curved_geometry}
    Assume the Lie group $\Gamma$ acts properly, smoothly and isometrically on a simply connected complete Riemannian manifold $M$ with nonpositive sectional curvature. Then $M$ with its $\Gamma$-action provides a model for $\classifyingspace{\cptfamily}{\Gamma}$.
\end{example}

\begin{example}[{\cite[Theorem 1]{Meintrup_Schick}}]\label{example:meinstrup_schick}
    Suppose, $\Gamma$ is a hyperbolic (discrete) group. 
    Then a barycentric subdivision of the \textit{Rips complex} of $\Gamma$ for sufficiently high $\delta > 0$ for a word metric on $\Gamma$ provides a finite $\Gamma$-CW model for $\classifyingspace{\finite}{\Gamma}$. In particular, the $\Gamma$-space $\classifyingspace{\finite}{\Gamma}$ is compact.
\end{example}

More examples for geometrically interesting models for the universal spaces for proper actions can be found in the survey \cite{Lueck_Survey_Classifying_Spaces} and the references therein.
For the next statement, recall that for a subgroup $H \le \Gamma$, its \textit{normaliser} is defined as $N_\Gamma H \coloneqq \{g \in \Gamma \mid g H g^{-1} = H\}$ and its \textit{Weyl group} as $W_\Gamma H \coloneqq N_\Gamma H  / H$.

\begin{thm}[L\"uck-Weiermann's decomposition, {\cite[Thm. 2.3, Cor. 2.10]{LWClassifying}}]
    \label{thm:lueck_weiermann}
    Let $\Gamma$ be a discrete group such that each nontrivial finite subgroup is contained in a unique maximal finite subgroup.
    Let $\mathcal{M}$ be a set of representatives of conjugacy classes of maximal finite subgroups of $\Gamma$. 
    Then the square
    \begin{equation}\label{diag:lueck_weiermann_pushout}\tag{$\square$}
    \begin{tikzcd}
        \coprod_{F \in \mathcal{M}} \ind_{N_\Gamma F}^\Gamma \myuline{EN_\Gamma F} \ar[r] \ar[d] 
        & \udl{E\Gamma} \ar[d] 
        \\
        \coprod_{F \in \mathcal{M}} \ind_{N_\Gamma F}^\Gamma \classifyingspace{\finite}{N_\Gamma F} \ar[r] 
        & \classifyingspace{\finite}{\Gamma}
    \end{tikzcd}
    \end{equation}
    is a pushout in $\spc_\Gamma$.
\end{thm}
\begin{rmk}
    \label{rmk:Lueck_Weiermann_example}
    To prove \cref{thm:lueck_weiermann} one checks that \cref{diag:lueck_weiermann_pushout} is a pushout on all fixed points by distinguishing the three cases $H = 1$, $H \neq 1$ finite and $H$ infinite.
    Examples of groups for which it applies are discrete groups $\Gamma$ for which there exists $\pi$ torsionfree and an extension
    \begin{equation*}
        1 \rightarrow \pi \rightarrow \Gamma \rightarrow C_p \rightarrow 1.
    \end{equation*} 
    If $\Gamma$ is assumed pseudofree, then we see that $\classifyingspace{\finite}{\Gamma}^{>1}$ is discrete.
\end{rmk}

\begin{example}
    Let us give an illustrative geometric example.
    Consider the group $\Gamma = p3$, the symmetry group of the wallpaper depicted in \cref{fig:p3}, with its action on the euclidean plane.
    This action is isometric and has finite stabilisers, so by \cref{ex:manifold_models_from_nonpositively_curved_geometry} it is a model for the universal space of proper action of $p3$.

    Note that the translations form a normal torsionfree subgroup of $p3$ of index 3.
    We can apply \cref{thm:lueck_weiermann} to obtain that the subspace of the plane with nontrivial $\Gamma$-isotropy is $\Gamma$-homotopy equivalent to $\coprod_{F} \ind_{N_\Gamma F}^\Gamma \classifyingspace{\finite}{N_\Gamma F}$. 
    Now from the picture it is easy to read off that the singular part is indeed discrete and consists of three $\Gamma$-orbits. 
    We conclude that $\Gamma$ has precisely three conjugacy classes of nontrivial finite subgroups, and each such nontrivial finite $F \subset \Gamma$ satisfies $N_\Gamma F = F$. 
\end{example}

\begin{figure}[H]
    \centering
    \subfloat{\includegraphics[width=0.2\textwidth]{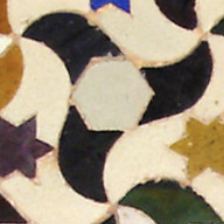}}\hspace{0.1\textwidth}
    \subfloat{\includegraphics[width=0.2\textwidth]{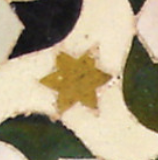}}\hspace{0.1\textwidth}
    \subfloat{\includegraphics[width=0.2\textwidth]{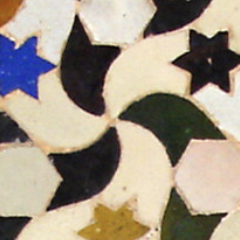}}
    \caption{Three points with $C_3$-symmetry, corresponding to three conjugacy classes of finite subgroups of $p3$.}
\end{figure}

\subsection{Equivariant Poincar\'e duality for groups}

For the following, a closed subgroup $\pi \subset \Gamma$ of a Lie group is \textit{cocompact} if the topological space $\pi \backslash \Gamma$ is compact. If $\pi \subset \Gamma$ is normal, this is equivalent to $\Gamma/\pi$ being a compact Lie group.

\begin{defn}
    Let $\Gamma$ be a Lie group and let $\pi \subset \Gamma$ be a cocompact torsionfree discrete normal subgroup. We write
    \[ \quotientspace{\Gamma}{\pi} \coloneqq \pi \backslash \classifyingspace{\cptfamily}{\Gamma} \in \spc_{\Gamma/\pi} \]
    for the quotient of $\classifyingspace{\cptfamily}{\Gamma}$ by the action of $\pi$.
\end{defn}

The following definition is supposed to capture the homological properties of Lie groups $\Gamma$ that admit a cocompact smooth manifold model for $\classifyingspace{\cptfamily}{\Gamma}$ (and a discrete torsionfree normal cocompact subgroup).
If $\Gamma$ is torsionfree, it reduces to $\Gamma$ being a Poincar\'e duality group. Here, and only here we refer to Poincar\'e duality for compact Lie groups as also developed in \cite{HKK_PD}, but the reader mainly interested in discrete group actions can assume $\Gamma$ to be discrete throughout.

\begin{defn}\label{defn:gvd_groups}
    Let $\Gamma$ be a Lie group. Then $\Gamma$ is called a  \textit{genuine virtual Poincar\'e duality group} if
    it has a cocompact torsionfree normal subgroup, and if for any such cocompact torsionfree normal subgroup $\pi \subset \Gamma$, the $\Gamma/\pi$-space 
    $\quotientspace{\Gamma}{\pi}$ is a $\Gamma/\pi$-Poincar\'e space.
\end{defn}

\begin{prop}
    \label{prop:jumping_between_subgroups}
    Suppose $\Gamma$ is a Lie group with torsionfree cocompact normal subgroups $\pi, \pi' \subset \Gamma$ whose intersection is again cocompact. Then 
    \[ \quotientspace{\Gamma}{\pi} \text{ is $\Gamma/\pi$-Poincar\'e } \iff \quotientspace{\Gamma}{\pi'} \text{ is $\Gamma/\pi'$-Poincar\'e}. \]
\end{prop}

\begin{proof}
    If suffices to consider the case where $\pi \subseteq \pi'$. 
    Note that the normal subgroup $\pi' / \pi \subseteq \Gamma / \pi$ acts freely on $\quotientspace{\Gamma}{\pi}$.
    Applying \cite[Corollary 4.3.13]{HKK_PD}, we see that $\quotientspace{\Gamma}{\pi}$ is $\Gamma/\pi$-Poincar\'e if and only if $(\pi' / \pi) \backslash \quotientspace{\Gamma}{\pi} \simeq \quotientspace{\Gamma}{\pi'}$ is $\Gamma/\pi'$-Poincar\'e.
\end{proof}

Note that, in general, the intersection of two cocompact subgroups is not again cocompact, e.g. for $\mathbb{Z}, \sqrt{2}\mathbb{Z} \subset \mathbb{R}$.
In the case where $\Gamma$ is discrete, the intersection of two finite index subgroups is again finite, from which we obtain the following result.

\begin{cor}
    Suppose, $\Gamma$ is a discrete group. Then the following are equivalent.
    \begin{myenum}{(\arabic*)}
        \item The group $\Gamma$ is a  genuine virtual Poincar\'e duality group.
        \item There exists some torsionfree finite index normal subgroup $\pi \subset \Gamma$ such that the $\Gamma/\pi$-space $\quotientspace{\Gamma}{\pi}$ is $\Gamma/\pi$-Poincar\'e.
    \end{myenum}
\end{cor}

\section{Extensions by \texorpdfstring{$C_p$}{Cp}}\label{section:application_nielsel_realisation}

In this section, we study genuine virtual Poincar\'e duality groups sitting in an extension
\begin{equation}\label{eqn:extension_pseudofree}
    1 \to \pi \to \Gamma \to C_p \to 1
\end{equation}
more closely.
In \cref{sec:genuine_virtual_pd_characterisation} we will prove the characterisation from \cref{thm:genuine_virtual_duality_recognition};
in \cref{sec:property_H} we will use this to prove property (H) for pseudofree extensions.

\subsection{Characterisation of genuine virtual Poincar\'e duality groups}\label{sec:genuine_virtual_pd_characterisation}

\begin{lem}\label{lem:decompositon_of_fixed_points}
    Consider an extension of groups of the form \cref{eqn:extension_pseudofree} where $\pi$ is torsionfree. Write $\mathcal{M}$ for a complete set of representatives of the conjugacy classes of nontrivial finite subgroups of $\Gamma$.
    \begin{myenum}{(\arabic*)}
        \item If $F \le \Gamma$ is a nontrivial subgroup with $\pi \cap F = e$ (e.g. $F$ is finite), then the composition of $F\rightarrowtail\Gamma \to C_p$  is an isomorphism. In particular, $F$ is a maximal finite subgroup of $\Gamma$.
        \item There is an equivalence of spaces 
        \[ \quotientspace{\Gamma}{\pi}^{C_p} \simeq \coprod_{F \in \mathcal{M}} BW_{\Gamma}F. \]
    \end{myenum}
\end{lem}

\begin{proof}
    Point (1) follows as the kernel of $\Gamma \rightarrow C_p$ is torsionfree, so every finite subgroup of $\Gamma$ will inject into $C_p$.   For (2), observe that if $\pi$ acts freely on a $\Gamma$-space $\udl{Y}$, then the  map
    \[ (\pi \backslash \udl{Y}^{>1})^{C_p} \rightarrow (\pi\backslash \udl{Y})^{C_p} \]
    is an equivalence by \cref{fact:singular_part_and_quotient}.
    Thus, since applying $(\pi\backslash-)^{C_p}$ to the top row in the pushout square \cref{diag:lueck_weiermann_pushout}
    yields the map of empty spaces, we get an identification
    \[ 
        \quotientspace{\Gamma}{\pi}^{C_p} 
        \simeq \left(\pi \backslash \coprod_{F \in \mathcal{M}} \ind_{N_\Gamma F}^\Gamma \classifyingspace{\finite}{N_\Gamma F}\right)^{C_p} 
        \simeq \coprod_{F \in \mathcal{M}} (\pi \backslash \ind_{N_\Gamma F}^\Gamma \classifyingspace{\finite}{N_\Gamma F})^{C_p}. 
    \]
    Consider the surjective composition of group homomorphisms $N_\Gamma F \subset \Gamma \rightarrow C_p$. By \cref{obs:composition_of_inductions}, we get
    \[ \pi \backslash \ind_{N_\Gamma F}^\Gamma \classifyingspace{\finite}{N_\Gamma F} \simeq (\pi \cap N_\Gamma F) \backslash \classifyingspace{\finite}{N_\Gamma F}. \]
    Now note that, by (1), the only finite subgroups of $N_\Gamma F$ are $e$ and $F$.
    This implies $\classifyingspace{\finite}{N_\Gamma F} \simeq \infl_{W_\Gamma F}^{N_\Gamma F} \myuline{EW_\Gamma F}$.
    Also, the composition $\pi\cap N_\Gamma F \subset N_\Gamma F \rightarrow W_\Gamma F$ is an isomorphism. Indeed, it is injective, as it has at most finite kernel and $\pi$ is  torsionfree, and it is surjective as $N_\Gamma F$ is generated by $F$ and $\pi \cap N_\Gamma F$, and $F$ maps to zero in $W_\Gamma F$.
    We thus obtain
    \[ (\pi \cap N_\Gamma F) \backslash \classifyingspace{\finite}{N_\Gamma F} \simeq (\pi \cap N_\Gamma F) \backslash \infl^{W_\Gamma F}_{N_{\Gamma}F}\myuline{EW_\Gamma F} \simeq \infl_{C_p}^{e} BW_\Gamma F,  \]
    the second equivalence being an instance of \cref{fact:quotients_and_inflation} for $G=N_\Gamma F$, $Q= W_\Gamma F$ and $N = N_\Gamma F \cap \pi$.
    This finishes the proof of the second assertion.
\end{proof}

\begin{lem}
    \label{lem:fiber_of_fixed_point_inclusion}
    In the situation of \cref{lem:decompositon_of_fixed_points}, for each $x \colon \udl{*} \rightarrow \quotientspace{\Gamma}{\pi}^{C_p} $ there is a pullback 
        \begin{equation}
            \begin{tikzcd}
                \udl{T} \ar[r, "a"] \ar[d, "p"] & \quotientspace{\Gamma}{\pi}^{C_p} \ar[d, "j"] \\
                \udl{*} \ar[r, "b"] & \quotientspace{\Gamma}{\pi}
            \end{tikzcd}
        \end{equation}
        of $C_p$-spaces where $\udl{T}$ is a disjoint union of $C_p$-orbits and has exactly one fixed point.
        Furthermore, $x \simeq a s$ where $s \colon \udl{*} \to \udl{T}$ denotes the section coming from the fixed point of $\udl{T}$.
\end{lem}

\begin{proof}
    Note that the $\pi$-actions on $\classifyingspace{\finite}{\Gamma}$ as well as on $\classifyingspace{\finite}{\Gamma}^{>1}$ are free, so that by \cite[Lem. 2.2.38]{HKK_PD} we have a cartesian square
    \begin{equation}
    \label{eq:pullback_and_quotient}
        \begin{tikzcd}
            \classifyingspace{\finite}{\Gamma}^{>1} \ar[r] \ar[d] \ar[dr, phantom, near start, "\lrcorner"]& \inflated^{C_p}_{\Gamma} \quotientspace{\Gamma}{\pi}^{C_p} \ar[d, "j"]\\
             \classifyingspace{\finite}{\Gamma} \ar[r] & \inflated^{C_p}_{\Gamma} \quotientspace{\Gamma}{\pi}.
        \end{tikzcd}
    \end{equation}
    Now, recallin \cref{fact:singular_part_and_quotient}, we have $\quotientspace{\Gamma}{\pi}^{C_p} \simeq \pi \backslash \classifyingspace{\finite}{\Gamma}^{>1}$.
    The point $x$ gives rise to a map of $\Gamma$-spaces $h \colon \udl{*} \rightarrow \inflated^{C_p}_{\Gamma} \quotientspace{\Gamma}{\pi}^{C_p}$, and so by \cite[Lem. 2.2.39.]{HKK_PD}, we get a commuting square
    \begin{equation}\label{eqn:orbit_filler_commuting_square}
        \begin{tikzcd}
            \myuline{\Gamma/F} \ar[r, "f"] \ar[d] & \classifyingspace{\finite}{\Gamma}^{>1} \ar[d] \\
            \udl{*}\ar[r,"h = \pi \backslash f"] & \inflated^{C_p}_{\Gamma} \quotientspace{\Gamma}{\pi}^{C_p} 
        \end{tikzcd}
    \end{equation}
    for $F$ a subgroup with $\pi \backslash (\Gamma/F) \simeq *$ (so that $F$ is nontrivial) and $\pi \cap F = e$.
    From \cref{lem:decompositon_of_fixed_points} we now learn that the map $F \subset \Gamma \rightarrow C_p$ is an isomorphism. Hence, $F$ can be used to define a section $s \colon C_p \rightarrow \Gamma$.
    We will now construct the following diagram.
    
    Restricting the pullback \cref{eq:pullback_and_quotient} along $s$, we obtain the outer cartesian square of $C_p$-spaces in
    \begin{equation*}
        \begin{tikzcd}
            \res^\Gamma_{C_p} \coprod_{F' \in \mathcal{M} } \Gamma/N_{\Gamma}F' \ar[r, "\simeq"] \ar[d] \ar[dr, phantom, "(A)"] &\res_{C_p}^{\Gamma} \classifyingspace{\finite}{\Gamma}^{>1} \ar[r] \ar[d] \ar[dr, phantom, "(B)"]& \res_{C_p}^{\Gamma}\inflated^{C_p}_{\Gamma}\quotientspace{\Gamma}{\pi}^{C_p} \ar[d, "j"] \ar[dr, phantom, "(C)"] \ar[r, "\simeq"]&\quotientspace{\Gamma}{\pi}^{C_p} \ar[d] \\
            * \ar[r, "\simeq"]  &\res_{C_p}^{\Gamma}\classifyingspace{\finite}{\Gamma} \ar[r] & \res_{C_p}^{\Gamma}\inflated^{C_p}_{\Gamma} \quotientspace{\Gamma}{\pi} \ar[r, "\simeq"]&\quotientspace{\Gamma}{\pi}.
        \end{tikzcd}
    \end{equation*}
    Here, the lower equivalence in square $(A)$ comes from the observation that for any group $G$, the space $\myuline{E_G \finite}$ becomes equivariantly contractible when restricted to a finite subgroup. The upper equivalence combines \cref{thm:lueck_weiermann} with the observation, that $\res^\Gamma_{C_p} \ind_{N_\Gamma F}^\Gamma \classifyingspace{\finite}{N_\Gamma F} \rightarrow \res^\Gamma_{C_p} \ind_{N_\Gamma F}^\Gamma \udl{*} = \res^\Gamma_{C_p} \Gamma/N_\Gamma F$ is an equivalence, which is checked easiest by looking at the map on underlying spaces and $C_p$-fixed points.
    The square labeled $(B)$ is obtained by  restricting the pullback \cref{eq:pullback_and_quotient} along the section $s \colon C_p \rightarrow \Gamma$. By virtue of $s$ being a section of $\Gamma \rightarrow C_p$, and as inflation is nothing but restriction along a surjective group homomorphism, we have $\res^{C_p}_{\Gamma} \infl^{C_p}_{\Gamma} \simeq \id_{\spc_{C_p}}$, explaining the identifications in square $(C)$.  It now follows from \cref{lem:free_action_on_quotient_by_normaliser} below that the upper left corner has exactly one fixed point, as required. 
    
    To see the last statement that $x\simeq as$, observe that the section $s$ was chosen to identify $C_p$ with a specific finite subgroup $F \subset \Gamma$ having the property that the composition 
    \[ \udl{*} \xrightarrow{eF} \res^\Gamma_{C_p} \udl{\Gamma/F} \xrightarrow{\res^\Gamma_{C_p}f} \res^\Gamma_{C_p} \classifyingspace{\finite}{\Gamma}^{>1} \rightarrow \quotientspace{\Gamma}{\pi}^{C_p} \]
    is the point $x$ by \cref{eqn:orbit_filler_commuting_square}.
\end{proof}

\begin{lem}\label{lem:free_action_on_quotient_by_normaliser}
    Let $\Gamma$ be a group for which each nontrivial finite subgroup is contained in a unique maximal finite subgroup. 
    Then for maximal finite subgroups $F, F' \subseteq \Gamma$ we have that $F$ acts freely on $\Gamma / N_\Gamma F'$ if $F$ and $F'$ are not conjugate and $(\Gamma / N_\Gamma F')^F = *$ if $F$ and $F'$ are conjugate.
\end{lem}
\begin{proof}
    Suppose there is $f \in F \backslash e$ and $g \in \Gamma$ such that $f g N_\Gamma F' = g N_\Gamma F'$.
    Then $g^{-1} f g \in N_\Gamma F'$ so the subgroup generated by $F'$ and $g^{-1} f g$ is finite.
    By maximality of $F'$, we obtain $g^{-1} f g \in F'$.
    The nontrivial element $f$ lies in both maximal finite subgroups $F$ and $g F' g^{-1}$  of $\Gamma$ which agree by uniqueness.
    This shows that the $F$-action on $\Gamma / N_\Gamma F'$ is free if $F$ and $F'$ are not conjugate.

    In the other case it suffices to show that $(\Gamma/N_\Gamma F)^F = *$.
    One fixed point is clearly given by $e N_\Gamma F$.
    Suppose that there are $f \in F\backslash e$ and $g \in \Gamma$ such that $f g N_\Gamma F = g N_\Gamma F$.
    The argument from above shows that $F = g F g^{-1}$, i.e.  $g \in N_\Gamma F$.
\end{proof}

We now come to the proof of our main characterisation result for genuine virtual Poincar\'e duality groups coming from $C_p$--extensions.

\begin{thm}
    \label{thm:extension_by_Cp_are_gvd}
    Consider an extension of groups
    \[ 1 \rightarrow \pi \rightarrow \Gamma \rightarrow C_p \rightarrow 1 \]
    where $\pi$ is a Poincar\'e duality group, and assume that the $\Gamma$-space $\classifyingspace{\finite}{\Gamma}$ is compact. Then the following are equivalent.
    \begin{myenum}{(\arabic*)}
        \item The group $\Gamma$ is a genuine virtual Poincar\'e duality group.
        \item For each nontrivial finite subgroup $F \subset \Gamma$, the Weyl group $W_\Gamma F$ is a Poincar\'e duality group.
    \end{myenum}
\end{thm}

\begin{proof}
    First of all, note that since $\pi \backslash(-)$ preserves compact objects -- it admits a right adjoint with a further right adjoint -- the $C_p$-space $\quotientspace{\Gamma}{\pi} =  \pi\backslash\udl{E\finite}_{\Gamma}$ is compact. As such,  $\quotientspace{\Gamma}{\pi}^{C_p}$ is also  compact. 
    
    To prove that (1) implies (2), recall that by \cref{lem:decompositon_of_fixed_points} we get an equivalence
    \begin{equation}\label{eqn:coproduct_decomposition}
        \quotientspace{\Gamma}{\pi}^{C_p} \simeq \coprod_{F\in \mathcal{M}} BW_\Gamma F
    \end{equation}
    where $F$ runs through a set $\mathcal{M}$ of representatives of the conjugacy classes of nontrivial finite subgroups. If a space is Poincar\'e, then each individual component is Poincar\'e. So we learn that $W_\Gamma F$ is a Poincar\'e duality group for each $F \in \mathcal{M}$. As conjugate subgroups have isomorphic Weyl groups, this implies that the conclusion holds for each nontrivial finite subgroup $F$.

    To prove that (2) implies (1), since $\quotientspace{\Gamma}{\pi}^{C_p}$ is compact by the first paragraph, there must only be finitely many components in the decomposition \cref{eqn:coproduct_decomposition}. By the hypothesis of (2), each component is Poincar\'e, implying that $\quotientspace{\Gamma}{\pi}^{C_p}$ is (nonequivariantly) Poincar\'e.
    We are thus in the situation of \cref{thm:characterization_of_Cp_PD}.
    To apply it, we have to show that for all $x \in \quotientspace{\Gamma}{\pi}^{C_p}$, the cofiber of the map $D_{\quotientspace{\Gamma}{\pi}^{C_p}}(x)^e \rightarrow j_! D_{\quotientspace{\Gamma}{\pi}^{C_p}}(x)^e$ lies in the perfect subcategory $\spectra^{BC_p}_{\induced} \subset \spectra^{BC_p}$ generated by induced spectra, where $j \colon \quotientspace{\Gamma}{\pi}^{C_p} \to \quotientspace{\Gamma}{\pi}$ denotes the inclusion.
    
    For ease of notation, let us write $\udl{X} \coloneqq \quotientspace{\Gamma}{\pi}$ in the following.
    From \cref{lem:fiber_of_fixed_point_inclusion}, we get a cartesian square of $C_p$-spaces
    \begin{equation}
            \begin{tikzcd}
                \udl{T} \ar[r, "a"]\ar[dr,phantom, near start, "\lrcorner"] \ar[d,"p"] & \udl{X^{C_p}} \ar[d,"j"] \\
                \udl{*} \ar[r, "b"] \ar[u, bend left, dashed, "s"] & \udl{X}
            \end{tikzcd}
    \end{equation}
    where $\udl{T} = \udl{*} \coprod \udl{S}$ and $\udl{S}$ is a disjoint union of free $C_p$-orbits where, furthermore, the point $x$ corresponds to the image of the composite $a s$, where $s \colon \udl{*} \to \udl{T}$ denotes the section coming from the fixed point of $\udl{T}$.

    Now observe that the map $D_{\udl{X^{C_p}}}(x) \to j^* j_! D_{\udl{X^{C_p}}}(x)$ identifies with the map $s^* a^* D_{\udl{X^{C_p}}} \to s^* p^* p_! a^* D_{\udl{X^{C_p}}} \simeq \colim_{\udl{T}} a^* D_{\udl{X^{C_p}}}$ induced by the unit $\id \to p^* p_!$.
    The decomposition $\udl{T} \simeq \udl{S} \coprod \udl{*}$ now provides a splitting 
    \[\colim_{\udl{T}} a^* D_{\udl{X^{C_p}}} \simeq s^* a^* D_{\udl{X^{C_p}}} \oplus \colim_{\udl{S}} a^* D_{\udl{X^{C_p}}}|_{\udl{S}}\]
    and the induced map
    \begin{equation}\label{eq:splitting}
        s^* a^* D_{\udl{X^{C_p}}} \to s^* a^* D_{\udl{X^{C_p}}} \oplus \colim_{\udl{S}} a^* D_{\udl{X^{C_p}}}|_{\udl{S}} 
    \end{equation}
    is an equivalence on the first component by functoriality of colimits.
    The $C_p$-spectrum $\colim_{\udl{S}} a^* D_{\udl{X^{C_p}}}|_{\udl{S}}$ is induced as $\udl{S}$ is free.
    Together this shows that the cofiber of the map $s^*a^* D_{\udl{X^{C_p}}} \to s^*a^* j^*j_! D_{\udl{X^{C_p}}}$ lies in the subcategory $\spectra^{BC_p}_{\induced} \subset \spectra^{BC_p}$ as was to be shown.
\end{proof}

\begin{rmk}
    Instead of explicitly identifying the map $\epsilon \colon D_{\udl{X^{C_p}}}(x) \to j^* j_! D_{\udl{X^{C_p}}}(x)$ in the last step of the argument above, one can also finish the proof using the following trick. Using the splitting \cref{eq:splitting}, one can reduce to showing that the cofiber of a selfmap $f$ of the invertible spectrum $D_{\udl{X}^{C_p}}(x)$ is induced. As $j$ is an equivalence on $C_p$-fixed points, one sees that $\Phi^{C_p}(\epsilon)$ is an equivalence. This implies that the selfmap $f$ in question also is an equivalence on geometric fixed points. The Burnside congruences show that $\cofib(f)^{e}$ is $n$-torsion for $n$ congruent to $1$ mod $p$, in particular $n$ coprime to $p$. But every compact spectrum with $C_p$-action which is $n$-torsion for $n$ coprime to $p$ vanishes in $\stmodSmall(C_p)$, so it is induced.
\end{rmk}

\subsection{Condition (H)}\label{sec:property_H}

In this section we will prove an abstract version of Condition (H) for general $C_p$-Poincar\'e spaces with discrete fixed points and see how this implies Condition (H) from \cref{nota:condition_H}.
Essential for this is the theory of singular parts and equivariant fundamental classes, especially the gluing class, introduced in \cite[\textsection 4.5]{HKK_PD}.
Let us recall the  relevant notions and constructions here.

\begin{cons}[{\cite[Cons. 4.5.4]{HKK_PD}}]\label{cons:gluing_class}
    For $\xi\in \func_G(\underline{X},\myuline{\spectra})$, there is a preferred map $(\uniquemap{X}_! \xi)^{hG} \to \Sigma (\uniquemap{X
    ^{>1}}_! \varepsilon^* \xi)_{hG}$.
    It is defined as the blue composite in the  commuting diagram
    \begin{equation}\label{diag:black_magic_Atiyah-Bott}
    \begin{tikzcd}
        (\uniquemap{X^{>1}}_!\varepsilon^*\xi)_{hG}\rar\dar\ar[dr, phantom, very near end, "\ulcorner"] 
        & (\uniquemap{X^{>1}}_!\varepsilon^*\xi)^{hG} \rar\dar
        & (\uniquemap{X^{>1}}_!\varepsilon^*\xi)^{tG} \rar[color =blue]\dar["\simeq",color =blue]
        & \Sigma(\uniquemap{X^{>1}}_!\varepsilon^*\xi)_{hG}
        \\
        (\uniquemap{X}_!\xi)_{hG}\rar\dar 
        & (\uniquemap{X}_!\xi)^{hG}\rar[color =blue]\dar[color =red] 
        & (\uniquemap{X}_!\xi)^{tG}
        \\
        Q_{hG} \rar["\simeq",color =red]\dar[color =red]
        & Q^{hG} 
        \\
        \Sigma(\uniquemap{X^{>1}}_!\varepsilon^*\xi)_{hG}
    \end{tikzcd},
    \end{equation}
    where the horizontal and vertical sequences are cofiber sequences and where we used the shorthand $Q\coloneqq \cofib \left( \uniquemap{X^{>1}}_! \varepsilon^* \xi \rightarrow \uniquemap{X}_! \xi \right)$.    
    By \cite[Lemma B.0.1]{HKK_PD}, the red and blue composite in \cref{diag:black_magic_Atiyah-Bott} are equivalent up to a sign.
\end{cons}

\begin{cons}[Gluing classes, {\cite[Cons. 4.5.5]{HKK_PD}}]\label{cons:generalised_gluing_class}
    Let $\underline{X} \in \spc_G$ be a $G$-Poincar\'e space with dualising spectrum $D_{\underline{X}} \in \func_G(\underline{X}, \myuline{\spectra})$ and collapse map $c \colon \sphere_G \rightarrow \uniquemap{X}_! D_{\underline{X}}$.
    The \textit{gluing class} of $\udl{X}$ is defined to be the composite
    \begin{equation*}
        \sphere 
        \xlongrightarrow{\canonical} \sphere_G^{hG} 
        \xlongrightarrow{c^{hG}} (\uniquemap{X}_! D_{\underline{X}})^{hG}
        \longrightarrow \Sigma(\uniquemap{X^{>1}}_! \varepsilon^* D_{\underline{X}})_{hG},
    \end{equation*}
    where the last map is the blue composite from \cref{diag:black_magic_Atiyah-Bott}.
    The \textit{linearised gluing class} is obtained by postcomposing with the map induced by $\sphere \to \bbZ$
    \begin{equation*}
        \Sigma(\uniquemap{X^{>1}}_! \varepsilon^* D_{\underline{X}})_{hG} 
        \rightarrow \Sigma(\uniquemap{X^{>1}}_! \varepsilon^* D_{\underline{X}} \otimes \bbZ)_{hG}.
    \end{equation*}
\end{cons}

We now specialise to the case $G=C_p$. Recall from \cref{cons:singular_part} that for $\udl{X}\in\spc_{C_p}$, we have $\udl{X}^{>1}\simeq \udl{X^{C_p}}$ coming from the fact that $X^{>1}\simeq X^{C_p}$.

\begin{prop}[Abstract condition (H)]\label{prop:property_H}
    Let $\udl{X}$ be a $C_p$-Poincar\'e space with discrete fixed points such that each component of $X^e$ has positive dimension.
    Then the linearised gluing class
    \[ 
    \sphere 
    \rightarrow \Sigma(\uniquemap{X^{>1}}_! \varepsilon^* D_{\udl{X}} \otimes  \bbZ)_{hC_p} 
    \xleftarrow{\simeq} \bigoplus_{y \in X^{C_p}} \Sigma(D_{\udl{X}}(y) \otimes \bbZ)_{hC_p} 
    \]
    maps a generator of $\pi_0(\sphere) \simeq \bbZ$ to a generator of $\pi_0 \Sigma(D_{\udl{X}}(y) \otimes \bbZ)_{hC_p} \simeq \bbZ/p$ in each summand.
\end{prop}

\begin{recollect}[Invertible $C_p$-spectra and group (co)homology]\label{rec:invertible_Cp_spectra_Z}
    For the proof of \cref{prop:property_H}, recall the following facts about the homology of invertible $C_p$-spectra.
    Recall that for an abelian group with $G$-action $M$, writing $M[d]$ for the corresponding object in $\module_\bbZ^{BG}$ concentrated in degree $d$, we have
    \[ \pi_* M[d]_{hG} \simeq H_{*-d}(G;M), \hspace{2mm} \pi_{*} M[d]^{hG} \simeq H^{d-*}(G;M), \hspace{2mm} \pi_{*} M[d]^{tG} \simeq \widehat{H}^{d-*}(G;M).  \]
    For $E \in \Pic(\spectra_{C_p})$ there are integers $d^e$ and $d^f$ such that $E^{e} \otimes \bbZ \simeq \bbZ[d^e]$ after forgetting the $C_p$-action and $\Phi^{C_p}E \otimes \bbZ \simeq  \bbZ[d^f]$. We write $\bbZ$ for the trivial $C_p$-representation and $\bbZ^\sigma$ for the sign $C_2$-representation.
    From \cite[Sec. 8.1.]{krause2020picard} and elementary group homology computations we obtain \cref{tab:group_homologies}.
    In each case, if $d^e +1 \leq d^f$, another group homological computation together with \cref{tab:group_homologies} shows that the map
    \[ \pi_{d^f}(E^e \otimes \mathbb{Z})^{tC_p} \rightarrow \pi_{d^f-1}(E^e \otimes \mathbb{Z})_{hC_p} \]
    is an isomorphism between cyclic groups of order $p$.
    \begin{table}[h]
        \centering
        \begin{tabular}{|c|c|c|}
        \hline
             & $d^e -d^f$ even & $d^e - d^f$ odd \\
        \hline
         \multirow{2}{4em}{$p$ odd}    & $\pi_* E^{e} \otimes \bbZ \simeq \bbZ[d^e]$ & \multirow{2}{4em}{impossible} \\
         & $\pi_{d^f}(E^{e} \otimes \bbZ)^{tC_p} \simeq \bbZ/p$ &  \\
         
        \hline
         \multirow{2}{4em}{$p=2$}    & $\pi_* E^{e} \otimes \bbZ \simeq \bbZ[d^e]$ & $\pi_* E^{e} \otimes \bbZ \simeq \bbZ^{\sigma}[d^e]$ \\
         & $\pi_{d^f}(E^{e} \otimes \bbZ)^{tC_p} \simeq \bbZ/p$ & $\pi_{d^f}(E^{e} \otimes \bbZ)^{tC_p} \simeq \bbZ/p$ \\
        \hline
        \end{tabular}
        \caption{Homological information on invertible $C_p$-spectra}
        \label{tab:group_homologies}
    \end{table}
\end{recollect}

\begin{obs}\label{obs:commuting_square_of_geometric_fixed_points}
    Consider a map $f \colon \udl{Y} \to \udl{X}$ in $\spc_G$.
    We claim that there is a commuting  diagram of functors $\myuline{\spectra}^{\udl{X}} \to \myuline{\spectra}^{\Phi\widetilde{\proper}}$
    \begin{equation}
    \begin{tikzcd}\label{diag:Beck_Chevalley_unit_commutation}
        X_! \Phi \ar[r,"c", leftarrow] \ar[d, "\simeq"] 
        &  X_! f_! f^* \Phi \ar[d, "\simeq"]
        \\
        \Phi X_!\ar[r, "c", leftarrow]
        & \Phi X_! f_! f^*,
    \end{tikzcd}
    \end{equation}
    where the vertical maps are the Beck-Chevalley transformations, which are equivalences as the geometric fixed points functor $\Phi \colon \myuline{\spectra} \to \myuline{\spectra}^{\Phi\widetilde{\proper}}$ from \cite[Construction 2.2.31]{HKK_PD} preserves parametrised colimits, and the horizontal maps are induced by the adjunction counit $c \colon f_! f^* \to \id$.
    This follows immediately from \cite[Lemma 2.2.6]{kaifNoncommMotives}, again using that $\Phi$ preserves parametrised colimits. Importantly, the top map (and hence also the bottom map) is an equivalence by \cite[Lem. 4.2.3]{HKK_PD}.

    Notice also that, by naturality of Beck-Chevalley transformations, if we have a decomposition $\udl{Y} = \myuline{Y_1} \coprod \myuline{Y_2}$, the right vertical map in \cref{diag:Beck_Chevalley_unit_commutation} is compatible with this splitting.
\end{obs}

\begin{proof}[Proof of {\cref{{prop:property_H}}}.]
   By construction, the gluing class factors through the map
   \[ (\uniquemap{X^{>1}}_! \varepsilon^* D_{\underline{X}} \otimes \bbZ)^{tC_p} \rightarrow \Sigma (\uniquemap{X^{>1}}_! \varepsilon^* D_{\underline{X}} \otimes \bbZ)_{hC_p} \]
   which happens to be an isomorphism on $\pi_0$ using \cref{rec:invertible_Cp_spectra_Z} and that $\udl{X}^{>1}$ is a discrete $C_p$-space (so that $d^f =0$) with trivial $C_p$-action.
   It thus suffices to show that the gluing class maps to a generator in each summand of 
   \[\pi_0(\uniquemap{X^{>1}}_! \varepsilon^* D_{\underline{X}} \otimes  \bbZ)^{tC_p} \simeq  \bigoplus_{y \in X^{C_p}} \pi_0 (D_{\udl{X}}(y) \otimes  \bbZ)^{tC_p} 
   \simeq \bigoplus_{X^{C_p}} \bbZ/p. \]
   
    We have the following commutative diagram.
    \begin{center}
    \begin{tikzcd}
        \sphere \arrow[d, "\simeq"] \arrow[ddd, bend right, blue] \arrow[r, "c_{\udl{X}}^{C_p}"] 
        & \uniquemap{X^{>1}}_! D_{X^{C_p}} \arrow[d, "\simeq", color=violet]
        & \uniquemap{X^{>1}}_! D_{X^{C_p}} \arrow[d, "\simeq ", color=violet] \arrow[l, "\id", color=violet]
        &
        \\
        \Phi^{C_p} \sphere\arrow[d] \arrow[r, "\Phi^{C_p}c_{\udl{X}}"]
        & \Phi^{C_p} (\uniquemap{X}_! D_{\underline{X}}) \arrow[d] 
        & \Phi^{C_p} (\uniquemap{X^{>1}}_! \varepsilon^* D_{\underline{X}}) \arrow[l, "\simeq", color=violet] \arrow[d] \arrow[r] 
        & \Phi^{C_p} (\uniquemap{X^{>1}}_! \varepsilon^* D_{\underline{X}} \otimes  \infl^e_{C_p}\bbZ) \arrow[d] 
        \\
        \sphere^{tC_p} \arrow[r, "c_{\udl{X}}^{tC_p}"]
        & (\uniquemap{X}_! D_{\underline{X}})^{tC_p}
        & (\uniquemap{X^{>1}}_! \varepsilon^*D_{\underline{X}})^{tC_p} \arrow[l, "\simeq", blue] \arrow[r, blue]
        & (\uniquemap{X^{>1}}_! \varepsilon^* D_{\underline{X}} \otimes  \bbZ)^{tC_p}
        \\
        \sphere^{hC_p} \arrow[u] \arrow[r, "c_{\udl{X}}^{hC_p}", blue]
        & (\uniquemap{X}_! D_{\underline{X}})^{hC_p} \arrow[u, blue]
        &
        &
    \end{tikzcd}
    \end{center}
    The rightmost part is induced by the ring map $\sphere \to \bbZ$.
    The violet square is obtained from \cref{obs:commuting_square_of_geometric_fixed_points} applied to $f = \varepsilon\colon \udl{X^{C_p}}\simeq \udl{X}^{>1}\rightarrow\udl{X}$ and $D_{\udl{X}}$, using the equivalence $\Phi^{C_p} D_{\udl{X}} \simeq D_{X^{C_p}}$ from \cref{thm:HKK_Thm.4.2.9} and $\varepsilon^{C_p} = \id_{X^{C_p}}$.
    
    By definition,  the blue route recovers the gluing class.
    Following the upper route to the same object gives a class having the desired properties. Indeed, on $\pi_0$ the upper route reads
    \begin{equation*}
    \begin{tikzcd}
        \mathbb{Z} \arrow[r, "\Delta"] 
        & \bigoplus_{X^{C_p}}\mathbb{Z} 
        & \bigoplus_{X^{C_p}} \mathbb{Z} \arrow[l, equal] \arrow[d, "\simeq", color =red] 
        &                                                               
        \\
        &
        & \bigoplus_{X^{C_p}} \mathbb{Z} \arrow[r, "\simeq"]
        & \bigoplus_{X^{C_p}} \mathbb{Z} \arrow[d, "\text{proj}"] 
        \\
        &
        &
        & \bigoplus_{X^{C_p}} \mathbb{Z}/p                       
    \end{tikzcd}
    \end{equation*}
    where $\text{proj}$ refers to the sum of the projection maps $\bbZ \rightarrow \bbZ/p$. Here, all maps in sight preserve the individual summands of $\bigoplus_{X^{C_p}} \bbZ$: the only potentially nonobvious case is the vertical red map, which is dealt with in \cref{obs:commuting_square_of_geometric_fixed_points}. \qedhere
\end{proof}

\begin{proof}[Proof of \cref{thm:condition_H_automatic}]
    Recall that Condition (H) from \cref{nota:condition_H} asks about surjectivity of the upper composite in the diagram
    \begin{equation*}
    \begin{tikzcd}
        H^{\Gamma}_d(\classifyingspace{\finite}{\Gamma}, \classifyingspace{\finite}{\Gamma}^{>1}) \ar[r, "\partial"] \ar[d, "\simeq"]
        & H^{\Gamma}_{d-1} (\classifyingspace{\finite}{\Gamma}^{>1}) \ar[r] \ar[d, "\simeq"]
        & H^{F}_{d-1}(\udl{*}) \ar[d, "\simeq"]
        \\
        H^{\Gamma/\pi}_d(\pi \backslash \classifyingspace{\finite}{\Gamma}, \pi \backslash \classifyingspace{\finite}{\Gamma}^{>1}) \ar[r, "\partial"]
        & H^{\Gamma/\pi}_{d-1}(\pi \backslash \classifyingspace{\finite}{\Gamma}^{>1}) \ar[r]
        & H^{\Gamma/\pi}_{d-1}(\udl{*}),
    \end{tikzcd}
    \end{equation*}
    where the right horizontal maps are induced from the the projection onto the $F$-component in the splitting $\classifyingspace{\finite}{\Gamma}^{>1} = \coprod_{F' \in \mathcal{M}} \Gamma/F'$.
    We may thus equivalently show surjectivity of the lower horizontal composite.
    Denote $\udl{X} = \pi \backslash \classifyingspace{\finite}{\Gamma}$, which is $C_p$-Poincar\'e by \cref{thm:extension_by_Cp_are_gvd} since  in this case, $W_{\Gamma}F\cong \{e\}$ and $\myuline{E_{\Gamma}\finite}$ is compact by \cref{example:meinstrup_schick}.
    
    Now by  definition, the bottom composite in the diagram above is obtained by postcomposing $\cofib(X^{>1}_! \epsilon^* D_{\udl{X}} \otimes \bbZ \to X_! D_{\udl{X}} \otimes \bbZ )_{h C_p} \to \Sigma (X^{>1}_! \epsilon^* D_{\udl{X}} \otimes \bbZ)_{h C_p}$ with projection to a component of $X^{C_p}$. Thus, by \cref{prop:property_H} and with the alternative description of the gluing class via the red route in \cref{cons:gluing_class}, we  obtain the required surjectivity.
\end{proof}

\printbibliography

\end{document}